\documentclass[journal]{IEEEtran}

\usepackage{graphicx}
\usepackage{cite}
\usepackage{amsfonts, amsmath, amssymb}
\usepackage{amsthm}
\usepackage{tikz}
\usepackage{tikz-network}
\usepackage{standalone}
\usepackage{subcaption}
\usepackage{todonotes}
\usepackage{enumerate}
\usepackage{float}
\usepackage{xcolor}
\usepackage{empheq}

\usepackage{etoolbox}

\makeatletter

\def\@IEEEBIOskipN{0.75\baselineskip}

\expandafter\patchcmd\csname\string\IEEEbiography\endcsname
  {plus 1fil}
  {plus 0pt}
  {}
  {\PackageWarning{main_bioopt}
    {Could not patch IEEEbiography spacing}}

\patchcmd{\IEEEbiographynophoto}
  {4\baselineskip plus 1fil}
  {\@IEEEBIOskipN plus 0pt}
  {}
  {\PackageWarning{main_bioopt}
    {Could not patch IEEEbiographynophoto spacing}}

\makeatother

\newtheorem{definition}{Definition}[section]

\newtheorem{theorem}[definition]{Theorem}
\newtheorem{lemma}[definition]{Lemma}
\newtheorem{corollary}[definition]{Corollary}

\newtheorem{example}[definition]{Example}

\newcommand{\continuation}{??}

\newtheorem{assumption}[definition]{Assumption}

\AtEndEnvironment{example}{\hfill$\lozenge$}
\AtEndEnvironment{theorem}{\hfill$\triangle$}

\global\long\def\G{\mathcal{G}}
\global\long\def\E{\mathcal{E}}
\global\long\def\V{\mathcal{V}}

\global\long\def\R{\mathbb{R}}
\global\long\def\IM{\operatorname{im}}
\global\long\def\rk{\operatorname{rk}}
\global\long\def\ker{\operatorname{ker}}

\graphicspath{{figures/}}

\newif\ifimportant

\importantfalse

\begin{document}
\setcounter{dbltopnumber}{1}
\title{Formation control from the generic combinatorial viewpoint: 
 edge dynamics and directed sensing
}
\author{Louis Theran, Daniel
  Zelazo,~\IEEEmembership{Senior~Member,~IEEE}, Sean Dewar, and Bernd Schulze
  \thanks{L. Theran is with the School of Mathematics and Statistics,
    University of St Andrews, Scotland. D. Zelazo is with the Stephen
    B. Klein Faculty of Aerospace Engineering, Technion-Israel
    Institute of Technology, Haifa 3200003, Israel. S. Dewar is with the
    Numerical Analysis and Applied Mathematics (NUMA) unit at KU
    Leuven. B. Schulze is with
    the School of Mathematical Sciences, Lancaster University, UK.
    This work was supported by the Israel Science
    Foundation grant no. 453/24, UK Research and Innovation through
    the grants UKRI1112 and UKRI2397, and the Heilbronn Institute for
    Mathematical Research (HIMR) through the International Visitors Scheme.
    S. Dewar was additionally supported by the FWO grants G0F5921N
    (Odysseus) and G023721N, and by the KU Leuven grant iBOF/23/064.
    We thank the ICMS in Edinburgh, where this  project was initiated,
  for its hospitality.}
}

\maketitle

\begin{abstract}
We develop a geometric framework for distance-based formation control
that separates the evolution of inter-agent distances from its
realization by compatible node motions, reducing the stability problem
to the edge space. We show that local exponential convergence of the
edge dynamics implies local exponential convergence of the formation,
and that stability is certified by spectral properties of a linear
edge operator. We introduce a hierarchy of generic spectral
properties --- weak admissibility, admissibility, and strong
admissibility --- that provide necessary conditions for local
exponential stability. Specializing to directed sensing, we obtain a
necessary and sufficient spectral condition for local stability at an
arbitrary target, together with a quadratic sufficient certificate.
These conditions reveal that stability depends jointly on the graph
orientation and target geometry, and show that persistence is neither
necessary nor sufficient for local convergence. 
We show that every generically rigid graph admits an
admissible orientation and,
for acyclic
orientations, we give an exact combinatorial characterization of
admissibility. Finally, the quadratic certificate leads to a
semidefinite program for synthesizing stabilizing edge gains.
\end{abstract}


\section{Introduction}

The distance-constrained formation control problem has emerged as one
of the canonical problems for multi-agent coordination
\cite{AhnBook2020}.  The
objective is to steer a group of agents to a prescribed geometric
shape, specified by inter-agent distances, using only local sensing
and distributed control actions.  The classical solution, originally proposed
by \cite{krickBrouckeFrancis}, is a
gradient-based controller that uses the squared edge length
error as a potential in node space and follows the negative
gradient.  When the sensing graph is generically $d$-rigid,
the gradient controller is generically locally exponentially
stable, as shown by \cite{Sun2016}.  In fact, for certain classes of
rigid graphs, one
can design decentralized control laws for which all stable equilibria
correspond to desired formations, although the presence of
undesirable equilibria remains a fundamental challenge
\cite{Chen_SIAMJCO2017}. 

A complementary geometric viewpoint was developed in
\cite{Dorfler_TAC2010}, where attention was shifted from the agent
configuration to the space of feasible squared edge lengths, showing
that these measurements evolve on a manifold determined by the
rigidity constraints. This edge-space perspective is closely related
to the approach we will take here, where we further separate the
desired evolution of the edge measurements from its realization by
compatible motions of the agents. In the undirected setting, these
viewpoints are tied together by the symmetric gradient structure of
the closed-loop dynamics.

This structure fundamentally changes under \emph{directed sensing}.
The resulting interactions become
non-reciprocal, leading to closed-loop dynamics that can no longer be
considered gradient flows.  In particular, the symmetry and
positive-semidefinite structure that underlie the local convergence
theory for undirected formations are lost. The theory of
\emph{persistence} was developed to describe the geometric
implications of these directed constraints
\cite{Hendrickx_IJRNC2007}.  Persistence is fundamentally a kinematic notion, as it characterizes whether a directed assignment of distance constraints is compatible with maintaining a rigid formation under admissible agent motions. It does not, however, describe the dynamics induced by a particular feedback law or determine whether the desired formation is a stable equilibrium. 
Persistence and its variants have nevertheless played
a central role in the analysis and design of directed formations
\cite{Hendrickx_IJRNC2007,FidanYuAnderson2007,Anderson_CDC2007,Babazadeh_CDC2020}.

Consequently, passing from persistence to stability requires additional dynamical analysis. For minimally persistent planar formations with a
leader--first-follower structure, \cite{Yu_SIAMJCO2009} constructed
agent-dependent gains that locally stabilize the
formation. More recently, \cite{Zhang_JDSMC2018} established
nonlinear asymptotic stability of the directed controller for
minimally persistent formations constructed by directed
Henneberg-type vertex additions, in both two and three dimensions.
More generally, several works have
obtained convergence guarantees for special classes of directed
formations, including acyclic and leader–follower structures
\cite{CaoCDC2008,Yu_SIAMJCO2009, LiuDeQueiroz2020,Mirzaeedodangeh2024}, often by exploiting
additional structure in the interaction topology. At the same time,
it has been shown that, due to the decentralized information flow and
the topology of the configuration space, global stabilization is
generically impossible for broad classes of directed formations
\cite{BelabbasTAC2013}.

The central question addressed in this paper is: 
\begin{center}
\emph{Given an
arbitrary orientation of a rigid graph and a target formation,
when is the resulting directed distance-based controller
  locally stable, and what structural properties of the orientation
promote stability}?
\end{center}
We adopt a geometric
viewpoint based on the relationship between node-space dynamics and
edge-space representations. Specifically, we study formation control
through an auxiliary system evolving on the manifold of feasible
distance measurements, which we then lift to agent dynamics. 
This construction leads to a general family of
compatible formation controllers that includes the classical gradient
controller, its directed counterpart, and a canonical minimum-norm
lift of the edge-space gradient flow.

The key advantage of our viewpoint is that it allows the convergence
problem to be studied directly in edge space. We show that local
exponential convergence of the edge dynamics implies local
exponential convergence of the agent configuration to a realization
congruent to the target. Moreover, local exponential stability is
characterized by the linearized edge operator restricted to the
tangent space of feasible edge measurements, with positive
definiteness of its quadratic form providing a stronger, readily
computable sufficient certificate.

The edge-based perspective also allows us to reveal the dependence between the
graph orientation and the target configuration in the directed
setting. The edge dynamics for the directed case are governed locally
by a generally nonsymmetric operator, and local stability may
therefore depend on the particular target even when the orientation
is fixed. To distinguish this target-dependent stability question
from generic properties of the orientation, we introduce three nested
spectral notions: \emph{weak admissibility}, \emph{admissibility},
and \emph{strong admissibility}. These respectively capture generic
full rank of the edge operator, invertibility of its restriction to
the feasible edge tangent space, and hyperbolicity of that
restriction. Local exponential stability at a generic target implies
strong admissibility. We also show that every generically rigid graph
admits  an admissible orientation.

For the special case of acyclic graphs, we show weak admissibility,
admissibility, strong admissibility, and local exponential stability
at every generic target are equivalent. We further show that this
structure extends beyond acyclic formations. A stable directed rigid
subformation may serve as the root of an acyclic extension. 
These results identify structural
classes of directed architectures for which stability can be inferred
from generic properties of the orientation rather than tested
separately at each target.  

We also compare our notions of admissibility to persistence, 
which we show does not characterize stability of the
feedback dynamics considered here.  Combined with our positive 
results, we see that stability of the directed controller 
depends on the spectral properties of the edge operator and 
not on the controller's information architecture.

Finally, since the quadratic
stability certificate is affine in the edge weights, we formulate a
semidefinite program for synthesizing stabilizing gains without
altering the underlying sensing architecture.



\paragraph*{Notations}
 We write \(\mathbb{R}\) for the set of real numbers and \(\mathbb{Q}\) 
 for the set of rational numbers. For a matrix $A$, $A^\top$ denotes 
 its transpose, $A^\dagger$ its
Moore--Penrose pseudoinverse, $\IM A$ its image, $\ker A$ its kernel,
and $\mathrm{rk}\, A$ its rank. We write $A\succ 0$ ($A\succeq 0$)
for symmetric positive definite (semidefinite). For a subspace $W$,
$W^\perp$ is its orthogonal complement. For a smooth map $F$,
$\mathrm{d}F_x$ denotes the differential at $x$. All norms are
Euclidean unless stated otherwise. For a manifold $\mathcal{M}$,
$T_x\mathcal{M}$ denotes the tangent space at $x$. The quadratic form
of a linear operator $T$ is $q_T(x):=\langle x,T(x)\rangle$, and for a
$T$-invariant subspace $V$, $T|_V:V\to V$ denotes the restriction.


\section{A Geometric Framework\\ for Formation Control}\label{sec.geoframework}

The starting point for this work is an ensemble of $n$ agents,
modeled by integrator dynamics,
\begin{align}\label{integrator}
  \dot p_i & = u_i,\quad i=1,\ldots,n,
\end{align}
where $p_i \in \R^{d}$ is the state (position) of agent $i$, and
$u_i \in \R^d$ is the control.  The aggregate position (control)
vector is denoted as $p =
\begin{bmatrix} p_1^\top & \cdots & p_n^\top
\end{bmatrix}^\top$ (similarly for $u$). We assume the agents
interact by an information exchange network described by the
undirected graph~$\G=(\V,\E)$.

The \emph{target configuration} $p^*$ has an associated set
of desired inter-agent square distances, $m_{ij}^*$, for each
$ij\in \E$; $m^* \in \R^{|\E|}$ is the aggregate target distance
vector. We  define the \emph{distance map},
$F \,:\,\mathbb{R}^{d n} \to \mathbb{R}^{ |\E|}$, by
\begin{align}\label{distance_map}
  F(p) =
  \begin{bmatrix} \cdots & \|p_j-p_i\|^2 & \cdots
  \end{bmatrix}^\top,
\end{align}
so that $F(p^*) = m^*$.  From now on, the squared edge length
measurement $F(p^*)$ of a target formation $p^*$ will be denoted $m^*$.
The distance-based formation control
problem is to drive an initial configuration $p^0$ to a point $q$
congruent and close to $p^*$ (such that $F(q)=m^*$),
using, at each time $t$, the information $m^*$ and  $m = F(p)$.
We introduce two
assumptions that we will use at various points in this paper.
\begin{assumption}\label{asump.drigid-reg}
  The graph $\mathcal G$ is $d$-rigid
  and the target configuration $p^*$ is a regular point (preimage of
  a regular value)
  of the rigidity map $F$.
\end{assumption}
\begin{assumption}\label{asump.drigid-geneic}
  The graph $\mathcal G$ is $d$-rigid
  and the target configuration $p^*$ is generic.
\end{assumption}
Here generic means that $p^*$ holds over the
open, dense complement of an unspecified exceptional set
that is defined by polynomials with rational coefficients.
Every generic configuration is regular, and if $p^*$ is
generic, then $F(p^*)$ is generic in the image of $F$.
A graph $\mathcal{G}$ is $d$-rigid if, for every configuration
$p^*$ so that $F(p^*)$ is smooth in the image of $F$, the
framework $(\mathcal{G},p^*)$ is locally rigid; in particular, when
$\mathcal{G}$ is $d$-rigid, the frameworks $(\mathcal{G},p^*)$
for regular $p^*$ will be rigid (see, e.g., \cite{Gortler_AJM2010}).

\subsection{A general family of formation controllers}

The formation control objective is specified entirely in terms of the
target distance vector $m^*$.
It is natural, therefore, to consider both the node positions $p$
and the distance measurements
$m$ as state-variables related by the constraint $F(p)=m$.  Rather
than designing the node
dynamics directly, we instead consider a more general problem of
constructing compatible node and
edge dynamics that evolve on the constraint manifold.

Since the constraint is that $m$ lies in the image of $F$, we
denote by $\mathcal{Q}$ the regular values of
the map $F$.  Because $F$ is a polynomial mapping defined over
$\mathbb{Q}$, the semialgebraic Sard theorem \cite[p. 234]{RAG}
implies that $\mathcal{Q}$ is a smooth
semialgebraic set that is open and dense in the image of $F$.
(The image of $F$ itself is not smooth.  We will see presently why
we use $\mathcal{Q}$ instead of the smooth locus of ${\rm Im}\, F$.)
We also note that $\mathcal{C} :=
F^{-1}(\mathcal{Q})\subseteq \R^{dn}$, the set of
regular points of $F$, is a smooth semi-algebraic set of dimension
$dn$. Let $R(p)$ denote the rigidity matrix of $(\G,p)$, so that
$$F(p)=R(p)p, \quad {\rm d}F_p=2R(p):T_p\mathcal{C}\rightarrow T_{F(p)}\mathcal{Q}.$$

At any point $m = F(p)$ in $\mathcal{Q}$,
the tangent space $T_m\mathcal Q$ is the vector space of differential
changes to the squared
edge lengths that can arise from the infinitesimal motions of
the agents, i.e.,
\[
  T_m\mathcal{Q} = \IM R(p).
\]

The manifold $\mathcal Q$ inherits the Euclidean inner product from
$\mathbb{R}^{|\mathcal E|}$, defining a Riemannian metric,
\[
  g_m(\xi_1,\xi_2) = \xi_1^\top \xi_2, \qquad \xi_1,\xi_2\in T_m\mathcal Q.
\]
Let $\Pi(m)\,:\,\R^{|\mathcal{E}|}\to
T_m\mathcal Q$ denote the orthogonal projector onto $T_m\mathcal Q$.
If $m$ is a regular value of $F$, for any $p\in F^{-1}(m)$, we have
\begin{equation}\label{eq: projector}
  \Pi(m) = R(p)R(p)^\dagger.
\end{equation}
(Note that if $m$ is smooth but not regular,
  then the formula for $\Pi(m)$ will be valid for some
$p$ in the fiber over $m$ but not others.) In particular, although
the expression on the right is written in terms of $p$, the resulting
projector depends only on $m$.

With the geometry specified above, the state-space for the node-edge
system may be defined as
\begin{equation}\label{ss-model}
  \mathcal{S} = \{(p,m)\in \mathcal{C}\times \mathcal{Q} : F(p) = m\},
\end{equation}
which is a smooth, irreducible, semi-algebraic set equipped with the natural
projections $\pi_1$ and $\pi_2$.
At any point $(p,m) \in \mathcal S$, its tangent space is
\[
  T_{(p,m)} \mathcal{S} = \{(\dot p,\dot m) \in T_p \mathcal{C}\times
  T_m \mathcal{Q} : \dot m = 2 R(p)\dot p\}.
\]

This setup leads to the central abstraction proposed in this paper.
Instead of designing control laws in the node space, we define a
formation controller by specifying compatible dynamics on the joint
node-edge state space $\mathcal S$.
The geometric relationship between the node and edge spaces is
illustrated in Figure~\ref{fig:geometric_framework}.

\begin{figure}[t]
  \centering
\includegraphics[
    width=0.5\columnwidth
]{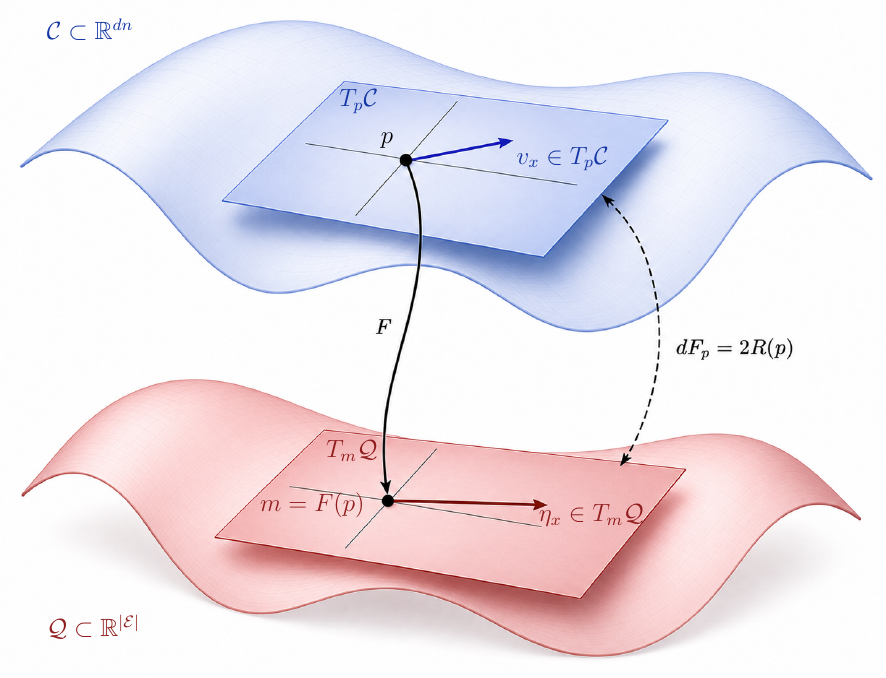}
  \caption{Geometric structure underlying compatible formation
    controllers. The distance map $F$ sends a configuration
    $p\in\mathcal C$ to its feasible edge measurement
    $m=F(p)\in\mathcal Q$. At the tangent level, node and edge velocities
    are related by
    $\eta_{(p,m)} = 2R(p)\nu_{(p,m)}$,
    which is the compatibility condition.} 
  \label{fig:geometric_framework}
\end{figure}

\begin{definition}[Compatible formation controller]\label{def.compatform}
  Let $\G$ be a generically $d$-rigid graph (which, along with the
    dimension $d$,
  determines, $\mathcal{Q}$ and $\mathcal{S})$.
  A \emph{compatible formation controller} on $\G$ is specified by a pair of
  state-dependent linear operators
  \[
    \nu_{(p,m)}:\R^{|\E|}\rightarrow(\R^d)^n,
    \qquad
    \eta_{(p,m)}:\R^{|\E|}\rightarrow\R^{|\E|},
  \]
  defined smoothly on $\mathcal S$,
  such that
  \[
    (\nu_{(p,m)}x,\eta_{(p,m)}x)\in
    T_{(p,m)}\mathcal S
  \]
  for every $(p,m)\in\mathcal S$ and every
  $x\in\R^{|\E|}$.
\end{definition}

The operator $\eta_{(p,m)}$ specifies the desired evolution of the
edge measurements, while $\nu_{(p,m)}$ realizes this edge motion in
node space. Since
\[
  (\nu_{(p,m)}x,\eta_{(p,m)}x)\in T_{(p,m)}\mathcal S,
\]
the induced node and edge velocities remain compatible with the
constraint $m=F(p)$. Equivalently,
\[
  \eta_{(p,m)} = 2R(p)\nu_{(p,m)},
\]
where the equality is understood as an identity of linear maps.

Given a target
formation $(\mathcal{G},p^*)$ with edge measurements
$m^*$ and a starting formation $(\mathcal{G},p^0)$, the
closed-loop dynamics are given by the initial value problem (IVP)
\begin{subequations}\label{eq:pm-dynamics}
  \begin{align}
    \dot p &= \nu_{(p,m)}(m^* - m) \label{eq:pm-dynamics-a}\\
    \dot m &= \eta_{(p,m)}(m^* - m) \label{eq:pm-dynamics-b}
  \end{align}
\end{subequations}
with initial conditions $p(0)=p^0$ and $m(0)=m^0=F(p^0)$.

For technical reasons, we will also allow controllers to be defined only
on an open neighborhood in $\mathcal{S}$ of $(p^*,m^*)$.  Controllers
in this family yield solutions $(p,m)\in \mathcal{S}$.  When
we discuss the \emph{edge dynamics} of one of these controllers, we
mean the flow $\pi_2(p,m)\in \mathcal{Q}$, and the node dynamics are,
similarly, $\pi_1(p,m)\in \mathcal{C}$.  The edge and node
dynamics have the
same smoothness as the full solution.
For notational ease, we use the shorthand $\eta := \eta_{(p,m)}$
and $\eta^*:=\eta_{(p^*,m^*)}$ (similarly for $\nu$).

Because distance measurements determine a rigid formation only up to congruence, stability of the node dynamics is understood relative to
the congruence class of the target.  This leads to a stability definition we adopt for the remainder of the work.
\begin{definition}[Local exponential stability]\label{def: locally exponentially stable}
A compatible controller is \emph{locally exponentially stable}
at a target configuration $p^*\in \mathcal{C}$ if there is a 
neighborhood $U\ni p^*$ such that, for all initial conditions 
$(p^0,F(p^0))$ with $p^0\in U$ the node dynamics converge
exponentially to a configuration $q\in U$ that is congruent 
to $p^*$.
\end{definition}

We now present our first main result showing that the edge dynamics
actually control the
convergence behavior of the node dynamics.

\begin{theorem}[Edge-to-node convergence]
  \label{thm:edge-to-node}
  Let Assumption~\ref{asump.drigid-reg} hold, and let
  $(\nu,\eta)$ be a compatible formation controller defined on an
  open neighborhood  $V\subseteq\mathcal S$ of $(p^*,m^*)$, where
  $\nu$ is smooth. Suppose further there exist constants
  $C,c>0$ such that every solution of
  \eqref{eq:pm-dynamics} with initial condition
  $(p^0,m^0)\in V$ satisfies
  \begin{equation}\label{edge_expconv}
    \|m(t)-m^*\|\le Ce^{-ct}\|m^0-m^*\|, \qquad t\ge0.
  \end{equation}
  Then there exists a neighborhood
  $U\subseteq V$ such that, for every initial
  condition $(p^0,m^0)\in U$, there is a configuration
  $q$ congruent to $p^*$ satisfying
  \[
    \|p(t)-q\|
    \le
    Ke^{-ct}\|m^0-m^*\|,
    \qquad t\ge0,
  \]
  where $K>0$ is independent of the initial condition. In particular, the controller is locally exponentially stable at $p^*$ in the sense of Definition~\ref{def: locally exponentially stable}.
\end{theorem}
Theorem \ref{thm:edge-to-node} shows that local exponential
convergence of the edge dynamics implies local exponential
convergence of the node dynamics to a configuration congruent to the
target. 

We will need the following technical ``trapping'' lemma to
bootstrap the main results.
\begin{lemma}\label{lem: trapping}
  Let Assumption~\ref{asump.drigid-reg} hold, and let
  $(\nu,\eta)$ be a compatible formation controller defined on an
  open neighborhood  $V\subseteq\mathcal S$ of $(p^*,m^*)$, where
  $\nu$ is smooth. Suppose further there exist constants
  $C,c>0$ such that every solution of
  \eqref{eq:pm-dynamics} with initial condition
  $(p^0,m^0)\in V$ satisfies \eqref{edge_expconv}
  for all times $t \ge 0$ such that the
  trajectory remains in $V$. Then there are neighborhoods
  $U\subseteq W\subseteq V$ of $(p^*,m^*)$ such that, for every solution of
  \eqref{eq:pm-dynamics} with initial condition
  $(p^0,m^0)\in U$, $(p(t),m(t))$ remains in $W$ for all times $t\ge 0$.
\end{lemma}
The proof is given in the Appendix.
\begin{proof}[Proof of Theorem \ref{thm:edge-to-node}]
  Let us fix a target $(p^*,m^*)\in \mathcal{S}$.  The goal is to
  find a neighborhood $U\subseteq \mathcal{S}$ of $(p^*,m^*)$ so that
  there is a constant $K > 0$ such that, for all initial conditions
  $(p^0,m^0)\in U$, there is  a configuration $q$
  congruent to $p^*$, such that
  $\|q - p(t)\| \le K{\rm e}^{-ct}$.

  It is assumed there are constants $C,c > 0$ and a neighborhood
  $ V\subseteq \mathcal{S}$
  of $(p^*,m^*)$ so that, for all initial conditions $(p^0,m^0)\in
  V$,
  $\| m(t)-m^*\| \le C{\rm e}^{-ct}\| m^0-m^*\| $.
  Because $\mathcal{G}$ is
  $d$-rigid and $p^*$ is a regular point of $F$,
  we can find a neighborhood $U_0$ of $p^*$ in $\mathcal{C}$
  such  that,
  if $q\in U_0$ and $F(q) = m^*$, then $q$ is congruent
  to $p^*$.  Let us now consider the full dynamics in the
  neighborhood $V_0 = (U_0 \times \mathcal{Q}) \cap V$ of the
  target $(p^*,m^*)$.

  By Lemma \ref{lem: trapping}, we can find neighborhoods $U\subseteq
  W\subseteq V_0$,
  such that, for all initial conditions in $U$, the trajectory
  remains in $W$ for all
  time.  Consider a sequence $t_i \to \infty$, and let $j>i$. Since the
  trajectory remains in $W$, we have
  \[
    \begin{aligned}
      \|p(t_j)-p(t_i)\|
      &\le \int_{t_i}^{t_j}\!\!\|\dot p(t)\|\,dt\le
      \int_{t_i}^{\infty}\!\! C'e^{-ct}\|m^0-m^*\|\,dt \\
      &\le \frac{C'}{c}e^{-ct_i}\|m^0-m^*\|,
    \end{aligned}
  \]
  where $C'=(\sup_W \|\nu\|)C$ (we have used that $\nu$ is smooth and
  defined on the closure of $W$).  Hence, the sequence $p(t_i)$
  is Cauchy, and has a limit $q$. A similar estimate gives
  \[
    \|q - p(t)\| \le {\frac{C'}{c}}e^{-ct}\|m^0-m^*\|=Ke^{-ct}\|m^0-m^*\|,
  \]
  so the convergence is exponential.  Since the whole state
  converges to $(q,m^*)\in \mathcal{S}$, we have $F(q) = m^*$ and
  $q\in U_0$, so $q$ is congruent to $p^*$.
\end{proof}
From Theorem \ref{thm:edge-to-node}, we get two certificates of
convergence for compatible controllers.
\begin{theorem}[Certificates of local exponential
  convergence]\label{thm: convergence certificates}
  Suppose that Assumption \ref{asump.drigid-reg} holds at a target $p^*$.
  Denote by $\eta^*$ the transformation $\eta_{(p^*,m^*)}$ and
  $T$ the tangent space $T_{m^*} \mathcal{Q}$.  Then, if either:
  \begin{enumerate}[(a)]
    \item the transformation $-\eta^*|_T$ is Hurwitz (i.e.~all eigenvalues
      of  $\eta^*|_T$ have positive real part);
    \item the quadratic form $q_{\eta^*}(v) = \langle
      v,\eta^*(v)\rangle$ is positive
      definite on $T$,
  \end{enumerate}
  then the controller is locally exponentially stable at $p^*$.  Moreover,
  (a) is necessary for local exponential stability.
\end{theorem}
\begin{proof}
  Let us first deal with sufficiency.  Because it is easier, we first prove (b).
  Let $A$ be the matrix of $\eta^*|_T$ in any basis.  The signature of
  $q_{\eta^*}$
  on $T$ is equal to that of $\frac{1}{2}(A + A^\top)$, so $-(A +
  A^\top) \prec 0$.
  From the Lyapunov LMI criterion, $-\eta^*|_T$ is Hurwitz.  Hence
  (b) implies (a).

  For (a), we use a Lyapunov method.  Let $A$ be the matrix of
  $\eta^*|_T$ in any basis of $T$.  Since $-A$ is
  Hurwitz, there exists a symmetric $P \succ 0$ such that
  $\frac{1}{2}(PA + A^\top P)\succ 0$. We can now define a
  positive-definite operator $\tilde P$ on $\mathbb{R}^{|\E|\times
  |\E|}$ by setting $\tilde P = P$ on $T$,
  and $\tilde P = I$ on $T^\perp$.  Consider the Lyapunov function
  $$W(m) = (m^*-m)^\top\tilde P (m^*-m).$$
  At $(p^*,m^*)$, the Lyapunov inequality implies that for some $c>0$
  we have the coercive estimate,
  $$v^\top\left(\tfrac{1}{2}(PA+A^\top P)\right)v \geq c \|v\|^2, \; v\in T.$$
  The main technical step is to extend this estimate from tangent
  vectors at $m^*$ to the nonlinear
  error $m^*-m$ for nearby states.  In particular, we show there exists
  a neighborhood $V$ of $(p^*,m^*)$
  such that, for $(p,m)\in V$,
  \begin{equation}\label{eq: coerce}
    q_{(p,m)}(m^* - m)\ge \alpha \|m^* - m\|^2,
  \end{equation}
  where $\alpha > 0$ is a constant independent of the
  state and $q_{(p,m)}:=q_{\widetilde P\eta_{(p,m)}}$ is the quadratic form 
  $$q_{(p,m)}(x)=x^\top\left(\tfrac{1}{2}(\tilde P\eta_{(p,m)}+\eta^\top_{(p,m)}\tilde P)\right)x.$$
  This implies that along trajectories in $V$,
  $$\dot W(m) = -2q_{(p,m)}(m^*-m) \leq -2\alpha \|m^*-m\|^2.$$
  Since $\tilde P\succ0$, $W(m)\leq \lambda_{\rm max}(\tilde
  P)\|m^*-m\|^2$, it follows that
  \begin{equation}\label{Wdot}
    \dot W(m)\leq -\tfrac{2\alpha}{\lambda_{\rm max}(\tilde P)} W(m).
  \end{equation}
  By Grönwall's inequality \cite[Lemma A.1]{khalil},
  $$W(m)\leq W(m^0)\mathrm{exp}\left(-\tfrac{2\alpha}{\lambda_{\rm
  max}(\tilde P)}t\right).$$
  Since $W(m)\geq \lambda_{\rm min}(\tilde P) \|m^*-m\|^2$, we obtain an
  estimate
  \[
    \|m^* - m\| \le Ce^{-ct}\|m^* - m^0\|, \quad 0\le t\le t^*,
  \]
  for trajectories initialized at $(p^0,m^0)\in V$.  Here $t^*$ is the
  supremal time that the trajectory remains in $V$.  Lemma
  \ref{lem: trapping} now implies that there is a neighborhood $U\subseteq V$
  of $(p^*,m^*)$ so that trajectories initialized at $(p^0,m^0)\in U$
  remain in $V$ for all time.  Local exponential convergence then follows
  from Theorem \ref{thm:edge-to-node}.

  What remains is to construct $V$ on which \eqref{eq: coerce} holds.
  First, since $q_{(p^*,m^*)}$ is positive definite, there is a $\beta
  > 0$ such that
  $q_{(p^*,m^*)}(v) > \beta \|v\|^2$ for all vectors $v\in T$.
  For $m\neq m^*$ in $\mathcal Q$, define the normalized secant
  \[
    e(m):=\frac{m^*-m}{\|m^*-m\|},
  \]
  and decompose $e(m)=v(m)+w(m)$ with $v(m)\in T$, and  $w(m)\in T^\perp$.
  Since $\mathcal Q$ is a smooth embedded manifold, as $m\to m^*$
  in $\mathcal Q$,
  it follows that $\|w(m)\|\rightarrow 0$ and $\|v(m)\|\rightarrow
  1$, uniformly.

  Since compatibility gives $\IM\eta^*\subseteq T$, and
  $\widetilde P$ preserves $T$ and $T^\perp$, we have $w^\top\widetilde
  P\eta^*x=0$ for every $w\in T^\perp$, and  $x\in\mathbb R^{|\E|}$.
  Therefore, writing $e=v+w$ as above,
  \[
    q_{(p^*,m^*)}(e)=q_{(p^*,m^*)}(v)+v^\top\widetilde P\eta^*w.
  \]
  Using the positive-definiteness estimate on $T$,
  \[
    q_{(p^*,m^*)}(e)
    \ge
    \beta\|v\|^2
    -
    \|\widetilde P\eta^*\|\,\|v\|\,\|w\|.
  \]
  Since $\|v\|\to1$ and $\|w\|\to0$ as $m\to m^*$ uniformly,
  there exists a neighborhood $U$ of $m^*$ in $\mathcal Q$
  such that
  \[
    q_{(p^*,m^*)}(e(m))\ge \frac{\beta}{2},
    \qquad
    m\in U,\quad m\neq m^*.
  \]

  By smoothness of $\eta$, after shrinking to a neighborhood
  $V\subseteq \mathcal S$ of $(p^*,m^*)$ with $m\in U$, we may assume that
  \[
    \bigl\|
    \widetilde P\eta_{(p,m)}
    -
    \widetilde P\eta^*
    \bigr\|
    <\frac{\beta}{4}.
  \]
  Hence, for the unit vector $e=e(m)$,
  \[
    \begin{aligned}
      q_{(p,m)}(e)
      &=
      q_{(p^*,m^*)}(e)
      +
      \left\langle
      e,
      \widetilde P
      \bigl(\eta_{(p,m)}-\eta^*\bigr)e
      \right\rangle \\
      &\ge
      \frac{\beta}{2}
      -
      \bigl\|
      \widetilde P
      \bigl(\eta_{(p,m)}-\eta^*\bigr)
      \bigr\| \ge \frac{\beta}{4}.
    \end{aligned}
  \]
  Therefore,
  \[
    q_{(p,m)}(m^*-m)
    \ge
    \frac{\beta}{4}\|m^*-m\|^2,
  \]
  for every $(p,m)\in V$. Thus \eqref{eq: coerce} holds with
  $\alpha=\beta/4$, completing the construction of $V$ and showing that
  (a) is sufficient.

  We now prove necessity. Suppose that the controller is locally
  exponentially stable at $(p^*,m^*)$. We show that
  $-\eta^*|_T$ is Hurwitz. Since $m^*$ is a regular value of $F$, there
  exists a smooth local section $\sigma:U\subset\mathcal Q\to\mathcal
  C$, of $F$, defined on a neighborhood $V$ of $(p^*,m^*)$ with compact
  closure so that $F(\sigma(m))=m$ and
  $\sigma(m^*)=p^*$.
  Fix $v\in T$. By smoothness of $\mathcal Q$, there exists a
  smooth curve $s:(-\varepsilon,\varepsilon)\to\mathcal Q$ such that
  $s(0)=m^*$, with $\dot s(0)=-v$. For each sufficiently small $t$,
  consider the solution of the controller with initial condition
  $(p_t^0,m_t^0)=(\sigma(s(t)),s(t))$. Let $e(\tau,t):=m^*-m(\tau;t)$
  denote the corresponding edge error.
  By construction, $e(0,t)=tv+o(t)$. Since the vector field is smooth,
  the solution depends smoothly
  on its initial condition. Differentiating the edge-error dynamics
  with respect to $t$ at $t=0$ therefore gives the variational
  equation along the equilibrium $(p^*,m^*)$.

  The edge error satisfies
  \[
    \dot e=-\eta_{(p,m)}e.
  \]
  At $(p^*,m^*)$ we have $e=0$, so in the linearization all
  terms involving derivatives of $\eta_{(p,m)}$ are multiplied by
  the equilibrium error and vanish. Hence, defining
  \[
    z(\tau):=
    \left.\frac{\partial e(\tau,t)}{\partial t}\right|_{t=0},
  \]
  we obtain   $\dot z=-\eta^*z,
    \;
    z(0)=v.$
    
  Compatibility gives $\IM\eta^*\subseteq T$, and hence $T$ is
  invariant under $\eta^*$. Therefore
  \[
    z(\tau)
    =
    \exp\!\left(-\tau\,\eta^*|_T\right)v.
  \]
  Equivalently, for every fixed $\tau\ge0$,
  \[
    e(\tau,t)
    =
    t\,\exp\!\left(-\tau\,\eta^*|_T\right)v+o(t)
    \qquad
    \text{as }t\to0.
  \]

  Since the controller is locally exponentially stable at
  $(p^*,m^*)$, after shrinking the neighborhood if necessary there
  exist constants $K,c>0$ such that, for every sufficiently small
  initial perturbation, the corresponding node trajectory converges
  exponentially to a realization congruent to $p^*$. Since $F$ is
  locally Lipschitz, there exists $C>0$ such that the associated
  edge error satisfies
  \[
    \|e(\tau,t)\|
    \le
    C e^{-c\tau}\|e(0,t)\|,
    \qquad \tau\ge0.
  \]

  Substituting the expansions above gives, for each fixed
  $\tau\ge0$,
  \[
    \left\|
    t\,\exp\!\left(-\tau\eta^*|_T\right)v+o(t)
    \right\|
    \le
    C e^{-c\tau}\|tv+o(t)\|.
  \]
  Dividing by $|t|$ and letting $t\to 0$ yields
  \[
    \left\|
    \exp\!\left(-\tau\eta^*|_T\right)v
    \right\|
    \le
    C e^{-c\tau}\|v\|.
  \]
  Since $v\in T$ was arbitrary,
  \[
    \left\|
    \exp\!\left(-\tau\eta^*|_T\right)
    \right\|
    \le
    C e^{-c\tau},
    \qquad \tau\ge0.
  \]
  Thus the linear system $\dot z=-\eta^*|_T z$
  is exponentially stable. Equivalently, every eigenvalue of
  $\eta^*|_T$ has positive real part, and therefore
  $-\eta^*|_T$ is Hurwitz. This proves necessity.
\end{proof}

\subsection{A model compatible controller}
Theorem \ref{thm: convergence certificates} reduces local convergence
of a compatible formation controller to
properties of its induced edge dynamics.  To develop a geometric
understanding of this result, we now consider
how one should ideally drive the edge measurements to the target
assuming the edge state $m \in \mathcal Q$
could be controlled directly. To this end, consider the edge
measurement vector $m\in\mathcal Q$ as the
state of the first-order system
\begin{equation}\label{eq:edge-integrator}
  \dot m = v,
\end{equation}
where $v\in\R^{|\E|}$ is an edge-space control input. This system
should not be interpreted as an
independent physical model of the formation. Rather, it specifies the
desired evolution of the
edge measurements, which must subsequently be realized by compatible
node motions.
Define the \emph{edge potential} function,
\[
  V_e(m) = \tfrac12\|m - m^*\|^2,
\]
which measures the squared deviation of the current
edge vector $m$ from the desired one $m^*$.
As we are interested in feasible motions that must remain
on $\mathcal Q$, the appropriate notion of the gradient is the
\emph{Riemannian gradient} on $(\mathcal Q,g)$, obtained by
projecting $\nabla V_e(m)= m - m^*$ onto the tangent space of $\mathcal{Q}$
at $m$,
\[
  \nabla_{\!\mathcal Q}V_e(m) = \Pi(m)\,\nabla V_e(m)
  = \Pi(m)\,(m - m^*).
\]
The \emph{Riemannian gradient flow} associated with $V_e$ is defined by
\begin{equation}\label{eq:riem-flow}
  \dot m = -\,\nabla_{\!\mathcal Q}V_e(m)
  = -\,\Pi(m)\,(m - m^*).
\end{equation}
A solution to \eqref{eq:riem-flow} describes a trajectory $m(t)\in\mathcal
Q$ whose velocity $\dot m(t)$ is the projection of the negative
Euclidean gradient of $V_e$ onto the tangent space $T_{m(t)}\mathcal Q$;
i.e., $m(t)$ follows
the steepest descent of $V_e$ with respect to the Riemannian metric~$g$.
\begin{corollary}[Exponential stability of the edge gradient flow]
  \label{prop:edge-gradient-stability}
  Let $m^*\in\mathcal Q$. Then $m^*$ is a locally exponentially
  stable equilibrium of the Riemannian gradient flow
  \eqref{eq:riem-flow}. In particular, there exist a neighborhood
  $U_{\mathcal Q}\subseteq\mathcal Q$ of $m^*$ and constants
  $C,c>0$ such that every solution with $m(0)\in U_{\mathcal Q}$
  satisfies
  \[
    \|m(t)-m^*\|
    \leq
    Ce^{-ct}\|m(0)-m^*\|,
    \qquad t\geq0.
  \]
\end{corollary}
This follows from the fact that $\Pi(m^*)$ is the identity on the
tangent space and the proof of Theorem \ref{thm: convergence certificates}
(the theorem itself deals with the combined dynamics). 

What was convenient about this case was that we have an intrinsic
description of the
edge dynamics that does not use any information about the node
state $p$.  This
means we can work on $\mathcal{Q}$ instead of $\mathcal{S}$.

To connect the squared edge length flow defined by the   Riemannian
gradient flow in \eqref{eq:riem-flow} to node dynamics, we make use of the
fact that $\mathcal{Q} = F(\mathcal{C})$.
The linear map ${\rm d}F_p =  2R(p)$ relates infinitesimal changes of the node positions $\dot p=u$
to infinitesimal changes of the squared edge lengths via
\[
  \dot m = {\rm d}F_p\, u = 2R(p)\,u.
\]
Hence, to produce a prescribed edge-space velocity $v^*\in
T_{F(p)}\mathcal Q$, we seek
a node-space velocity $u$ satisfying $${\rm d}F_p\,u=v^*.$$

Among the infinitely many $u$ that satisfy this relation,
a natural choice is the one with \emph{minimum Euclidean norm}:
\begin{align}\label{eq:node-opt}
  \min_{u\in\mathbb{R}^{dn}} \;\|u\|^2
  \quad\text{s.t.}\quad {\rm d}F_p\,u = v^*.
\end{align}
The unique minimizer of~\eqref{eq:node-opt} is obtained via the
Moore--Penrose pseudoinverse,
\[
  u^* = ({\rm d}F_p)^\dagger v^* = \tfrac{1}{2}\,R(p)^\dagger v^*.
\]
Consequently, the combined edge--node system can be viewed as a
two-level geometric control
construction.
Substituting $v^*=-\,\Pi(m)\,(m-m^*)$ yields the node-space dynamics
\begin{align}\label{model_controller_dynamics}
  \dot p = u^*
  = -\,\tfrac12\,R(p)^\dagger\,\Pi(m)\,(m-m^*).
\end{align}
The above construction gives us our first general compatible controller, with
\begin{subequations}\label{model_controller}
  \begin{empheq}[left=\empheqlbrace]{align}
    \nu &= \frac{1}{2}R(p)^\dagger, \label{model_controller_nu}\\
    \eta &= \Pi(m)=2R(p)\nu. \label{model_controller_eta}
  \end{empheq}
\end{subequations}
We refer to \eqref{model_controller}  as the \emph{model controller}.
This control,
while not distributed, and not intended for practical implementation,
provides a clear
geometric picture of what should be done to drive agents to a
desired formation.

\begin{example}[Classical distributed solution as a compatible controller]
We now show how the classical distributed solution to the formation control problem \cite{krickBrouckeFrancis} fits our framework.
%
The standard approach to derive this control is through the
negative-gradient flow of the potential function $V(p) =
\tfrac{1}{4}\|F(p)-m^*\|^2$, yielding
\begin{align}\label{grad_Control}
  u &= -\tfrac{\partial V(p)}{\partial p}=
  -R(p)^\top(R(p)p-m^*).
\end{align}
This control strategy can also be seen as a compatible formation
controller by setting
\begin{align}\label{distributed_model}
  \nu = R(p)^\top\text{ and } \eta = 2R(p)R(p)^\top.
\end{align}
These are  polynomial functions of the state $(p,m)$, and, hence,
are  rational over all of  $\mathcal{S}$. The column space of
$\eta$ is contained in the column space of $R(p)$, and
hence in $T_m \mathcal{Q}$.  By construction, we have
$2R(p)\nu = \eta$,   so the controller is well-defined.
The edge controller is no longer the Riemannian gradient:
$\eta$ is not a projection, nor is it constant over fibers
of $F$. However, the  matrix $2R(p)R(p)^\top$ is positive definite on
$T_m\mathcal Q$, which is sufficient for Theorem \ref{thm:
convergence certificates} to apply, yielding a new proof that the
controller is locally exponentially stable for every target $p^*$
such that $(p^*,F(p^*))\in \mathcal{S}$.
\end{example}
\section{Generic admissibility}\label{sec.admissibility}
The main result of Section \ref{sec.geoframework} showed that local
exponential convergence of a compatible formation controller 
can be certified through the edge dynamics.
The corresponding certificates, however, are pointwise conditions on
the operator $\eta^*|_{T_{m^*}\mathcal Q}$ at a particular target.
This motivates the study of
properties of $\eta$ that are \emph{intrinsic} to the controller,
rather than to a particular
target state $(p^*,m^*)$.  
We begin with a general
definition of a new notion we term \emph{admissibility} for
compatible controllers.
\begin{definition}[Admissibile compatible controller]\label{def: admissible}
  A compatible formation controller (Definition \ref{def.compatform}) defined
  by rational functions is called
  \begin{enumerate}[(a)]
    \item \emph{strongly admissible} if,
      for every generic
      target configuration $p^*\in \mathcal{C}$,
      $\eta^*|_{T_{m^*} \mathcal{Q}}$
      is hyperbolic (i.e.~all eigenvalues
      of $\eta^*|_{T_{m^*} \mathcal{Q}}$ have non-zero real part);
    \item \emph{admissible} if
      $\eta^*|_{T_{m^*} \mathcal{Q}}$ is invertible for every generic
      target $p^*$;
    \item \emph{weakly admissible} if the rank of $\eta^*$ is
      $dn - \binom{d+1}{2}$ for every generic target $p^*$.
  \end{enumerate}
\end{definition}
Pointwise, strong admissibility (all eigenvalues have non-zero real parts)
implies  admissibility (no eigenvalue is zero), implies weak
admissibility (since invertibility of $\eta^*|_{T_{m^*}\mathcal Q}$
requires $\eta^*$ to have full rank on $T_{m^*}\mathcal Q$).
Additionally, weak admissibility and admissibility agree if $\mathcal
G$ is minimally $d$-rigid.

\subsection{Genericity}
We now show how the admissibility conditions in Definition \ref{def:
admissible} can be verified generically.  In particular, weak
admissibility and admissibility can be certified at a single target,
while a Hurwitz condition at one target provides a certificate  for
strong admissibility. 
\begin{definition}[Genericity]\label{def: generic}
Let $P$ be a property of points in $\R^m$.  If there is a 
proper algebraic subset, $X\subseteq \R^m$, defined over $\mathbb{Q}$, 
such that either for all $y\notin X$, $P(y)$ holds or, 
$y\notin X$, $P(y)$ does not hold, then $P$ is called a 
\emph{generic property}.
\end{definition}
The motivation for this definition is that, as a proper algebraic set, 
$X$ is a closed, nowhere dense, subset of $\mathbb{\R}^m$.  Hence,
the points on which a generic property $P$ has atypical behavior are 
small in measure-theoretic, topological, or geometric terms.  

In what follows, we will often work with properties $P$ that 
are defined by the non-vanishing of a known collection of 
polynomials with rational coefficients. One example is 
that a structured matrix $A\in \R[{\bf t}]^{m\times n}$ 
with polynomial entries has rank at least $r$, where the polynomials 
are formal minors.  Another one is that a 
configuration lies in $\mathcal{S}$.  
In this setting, one can show that $P$ is automatically 
a generic property and that $P$ holds generically 
if it holds at a single point $y$.

\begin{theorem}[Admissibility is a generic property]\label{thm:
  admissible is generic}
  Fix a generically $d$-rigid graph $\G$ with $n$ vertices, a
  compatible formation controller, and a configuration $p^*\in
  \mathcal{C}$.  Denote
  by $m^*$ the edge lengths $F(p^*)$, so that $(p^*,m^*)\in \mathcal{S}$, by
  $\eta^*$ the linear operator $\eta_{(p^*,m^*)}$ and by $T$ the tangent
  space $T_{m^*} \mathcal{Q}$.  If either of the statements:
  \begin{enumerate}[(a)]
    \item $\eta^*$ has rank $N := dn - \binom{d+1}{2}$;
    \item $\eta^*|_T$ is invertible,
  \end{enumerate}
  holds at $p^*$ then it holds for every generic $p\in \mathcal{C}$.  Moreover,
  if $-\eta^*|_T$ is Hurwitz, then at every generic $p$, the
  transformation $\eta_{(p,F(p))}|_{T_{F(p)}\mathcal{Q}}$ is hyperbolic.
\end{theorem}
In the proof we will use some technical results about rational maps.
\begin{lemma}\label{lem: rational maps generic properties}
  Let $\alpha({\bf t}) : \R^m\to \R^m$ be a linear transformation that depends
  rationally on the variables ${\bf t} = (t_1, \ldots, t_m)$. Then
  \begin{enumerate}[(a)]
    \item if $\alpha({\bf t})$ is invertible for some ${\bf t}$, then it
      is invertible for generic ${\bf t}$;
    \item if, for some ${\bf t}$, no pair of eigenvalues of
      $\alpha({\bf t})$ sums to zero, then the same holds for
      generic ${\bf t}$.
  \end{enumerate}
\end{lemma}
We remark that we cannot strengthen the conclusion to say that
hyperbolicity of $\alpha$ is a generic property.  This is
why we use Hurwitz as the one-point certificate of
generic hyperbolicity.
\begin{lemma}\label{lem: rational maps compression}
  Let $\alpha({\bf t}) : \R^m\to \R^m$ be a linear transformation that depends
  rationally on the variables ${\bf t} = (t_1, \ldots, t_m)$. Then the
  restriction $\alpha({\bf t})|_{\IM \alpha({\bf t})}$ of $\alpha$ 
  to its image also depends rationally on ${\bf t}$.
\end{lemma}
The proofs of Lemma \ref{lem: rational maps generic properties} and
\ref{lem: rational maps compression} are given in the Appendix.

\begin{proof}[Proof of Theorem \ref{thm: admissible is generic}]
  Since $F$ is a polynomial map, the state
  $(p^*,m^*)=(p^*,F(p^*))$ depends polynomially on $p^*$.
  Since $\eta$ is rational in the state, it follows that $\eta^*$
  depends rationally on $p^*$. From compatibility, $\IM \eta^*\subseteq T$, and
  regularity of $m^*$ implies that $\dim T = N$.  Hence $\eta^*$ has rank $N$ if
  and only if the induced map $\mathbb{R}^{|\E|}/\ker \eta^*\to T$ is
  invertible.
  By Gaussian elimination, there is a complement $V$ to $\ker \eta^*$ such that
  the inclusion $V\hookrightarrow \mathbb{R}^{|\E|}$ is rational in the
  state.  This shows that the induced map is rational, and (a) follows from
  Lemma \ref{lem: rational maps generic properties}.

  Because rank cannot increase under restriction, statement (a)
  must hold whenever statement (b) does.  In that case, we have $T =
  \IM \eta^*$.
  Statement (b) then follows from Lemma \ref{lem: rational maps generic
  properties}
  applied to $\eta^*|_T$, which depends rationally on $p^*$ by
  Lemma \ref{lem: rational maps compression}.

  Generic hyperbolicity follows similarly, once we note that, if the
  real parts of every eigenvalue of $\eta^*|_T$ are positive, no
  pair of them can sum to zero.
\end{proof}

\subsection{Weak admissibility vs admissibility}
We can make a precise statement about the relationship
between weak admissibility and admissibility of a
compatible controller. In what follows, we will often use the
condition
\begin{equation}
  \label{eq: skew}
  \IM \alpha \cap \ker \alpha = \{0\}
\end{equation}
for a linear transformation $\alpha : \R^m\to \R^m$.  
\begin{theorem}[Weak admissibility upgrade]\label{thm: weak to strong}
  A weakly admissible compatible controller is admissible if and only if, at
  some generic target state $(p^*,m^*)\in \mathcal{S}$,
  $\eta^*$ satisfies \eqref{eq: skew}.
  Moreover, if, at $p^*$, \eqref{eq: skew} holds and the eigenvalues in
  the non-zero
  spectrum of $\eta^*$ have positive real parts,
  then $-\eta^*|_{T_{m^*} \mathcal{Q}}$ is
  Hurwitz, and the controller is strongly admissible.
\end{theorem}
\begin{proof}
  Denote by  $T$ the tangent space $T_{m^*} \mathcal{Q}$.
  By compatibility, $\IM \eta^* \subseteq T$, so weak admissibility
  corresponds to the
  statement that $\IM \eta = T$.
  Let $\pi$ be a linear projection $\R^{|\E|}\to T$ and $\pi^*$ the
  corresponding
  inclusion $\IM \eta = T\hookrightarrow \R^{|\E|}$.  Admissibility is
  the statement that
  $\pi\circ \eta\circ \pi^*$ is invertible.  By Theorem \ref{thm:
  admissible is generic}
  and genericity of $p^*$, we can check the statement pointwise at $p^*$.

  We will show that the controller is not admissible at $p^*$ if and only if
  \eqref{eq: skew} does not hold.  Suppose first that \eqref{eq: skew}
  does not hold.  By weak admissibility, $\IM \eta = T$, so there is a
  $v\in T$ such that $\eta^*(v) = 0$.  For this $v$, we have that
  $\pi^*(v) = v$, so
  $\pi\circ \eta\circ \pi^*(v) = 0$, certifying that the controller is
  not admissible at $p^*$.
  For the other direction, suppose the controller is not admissible at
  $p^*$.   Because the controller is not
  admissible, there is a non-zero $v\in T$ such that $\pi\circ
  \eta\circ \pi^*(v) = 0$. By weak
  admissibility $\IM \pi^* = \IM \eta$, so $w = \pi^*(v)\in \IM \eta$.
  By compatibility,
  $\pi\circ \eta(w) = \eta(w)$, so $w\in \ker \eta$; i.e.,
  \eqref{eq: skew} does not hold.

  In the case that \eqref{eq: skew} holds, because $\IM \eta = \IM
  \pi\circ \eta\circ \pi^*$,
  the two maps have the same non-zero eigenspaces.  Hence, if
  the eigenvalues in the non-zero spectrum of $\eta^*$
  have positive real parts, then $-\eta^*|_{T_{m^*} \mathcal{Q}}$ is Hurwitz.
\end{proof}
The condition \eqref{eq: skew} can be verified locally using a
flag of invariant subspaces.  To set up some notation, we let
$\alpha :\R^m\to \R^m$ be a linear operator and
$\{0\}  = V_0 \subseteq V_1 \subseteq \cdots \subseteq V_k = \R^m$
(we do not require the inclusions to be strict)
a flag of $\alpha$-invariant subspaces.  We define $W_j = V_j/V_{j-1}$
and denote by $\pi_j$ the quotient map and by $\alpha_j$ the induced
map $W_j\to W_j$, which is $\pi_j\circ \alpha\circ \pi^*_j$  (where the
inclusion $\pi^*_j$ is any linear section).  We say that the flag
$V$ satisfies the \emph{decoupling condition}
if $\IM \alpha \cap V_j = \alpha(V_j)$ for all $j$.
There are several equivalent formulations:
\begin{lemma}\label{lem: decoupling}
  Let $\alpha :\R^m\to \R^m$ be a linear operator and
  $V_i$ a flag of $\alpha$-invariant subspaces.  The
  following are equivalent:
  \begin{enumerate}[(a)]
    \item $V_i$ satisfies the decoupling condition;
    \item for all $j$, if $v\in V_j$ and $\alpha(v)\in V_{j-1}$,
      then $\alpha(v)\in \alpha(V_{j-1})$;
    \item for all $j$, $\rk \alpha|_{V_j} = \rk \alpha|_{V_{j-1}} +
      \rk \alpha_j$;
    \item $\rk \alpha = \rk \alpha_1 + \cdots + \rk \alpha_k$.
  \end{enumerate}
\end{lemma}
The proof is in the appendix. The point of the decoupling condition is that it
is a local property that implies the more global  \eqref{eq: skew}.  We
note that the rank inequality $\rk \alpha \ge \rk \alpha_1 + \cdots +
\rk \alpha_k$
holds unconditionally for any invariant flag.
\begin{theorem}[Invariant flag certificate of admissibility]\label{thm: flag}
  Let $V_{j,(p,m)}$ be a flag of $\eta_{(p,m)}$-invariant subspaces
  that depends rationally on the state, and suppose that the controller
  is weakly admissible. If, at some generic target $p^*$, the flag
  $V_{j,(p^*,m^*)}$ satisfies the decoupling condition, then the controller is
  admissible  if and only if each induced
  map $\eta_j^*$ satisfies~\eqref{eq: skew}.
  If, in addition, the nonzero spectrum of $\eta^*$ have positive
  real parts, then $-\eta^*|_{T_{m^*}\mathcal Q}$ is Hurwitz and the controller
  is strongly admissible.
\end{theorem}
The proof uses two further technical lemmas that are proved in the appendix.
\begin{lemma}\label{lem: flag}
  Let $\alpha({\bf t}) : \R^m\to \R^m$ be a linear transformation that
  depends rationally on ${\bf t}$ and $V_i({\bf t})$ a flag of
  $\alpha({\bf t})$-invariant subspaces that depends rationally
  on ${\bf t}$.  Then,  the spectrum of $\alpha({\bf t})$ is the union
  of the spectra of
  the induced maps $\alpha({\bf t})_j$, and the following are generic
  properties:
  \begin{enumerate}[(a)]
    \item $V_i({\bf t})$ satisfies the decoupling condition;
    \item the induced maps $\alpha({\bf t})_j$ satisfy \eqref{eq: skew}.
  \end{enumerate}
\end{lemma}
\begin{lemma}\label{lem: flag skew}
  Let $\alpha :\R^m\to \R^m$ be a linear operator and
  $V_i$ a flag of $\alpha$-invariant subspaces satisfying the
  decoupling condition.  Then \eqref{eq: skew} holds for
  $\alpha$ if and only if it holds for each induced map $\alpha_j$.
\end{lemma}
\begin{proof}[Proof of Theorem \ref{thm: flag}]
  By Theorems \ref{thm: admissible is generic} and \ref{thm: weak to strong}
  and Lemma \ref{lem: flag},  all the properties in the hypothesis
  are generic, and so may be checked pointwise at $p^*$.  The result now
  follows from  Lemma
  \ref{lem: flag skew} and then Theorem \ref{thm: weak to strong}.
\end{proof}
\begin{figure*}[!t]
  \centering
  \begin{subfigure}[c]{0.28\textwidth}
    \centering
\includegraphics[
  width=\linewidth
]{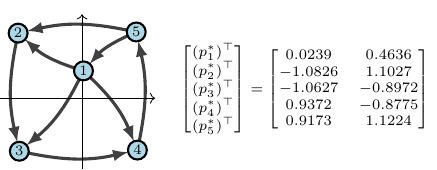}
    \caption{Directed target.}
    \label{fig:ex3_graph}
  \end{subfigure}
  \hfill
  \begin{subfigure}[c]{0.2\textwidth}
    \centering
    \includegraphics[width=\linewidth]
    {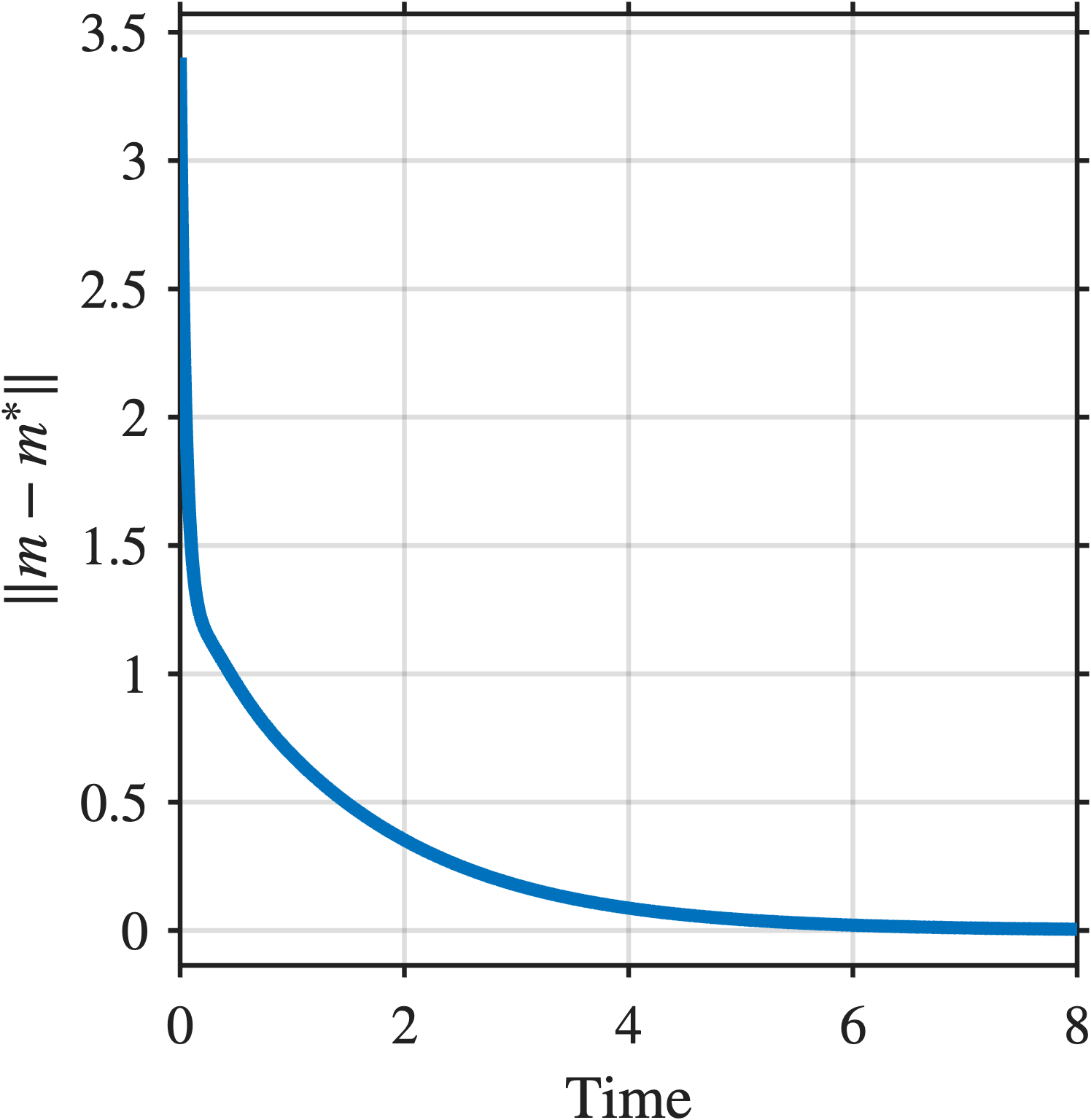}
    \caption{Edge-error norm.}
    \label{fig:ex3_edge}
  \end{subfigure}
  \hfill
  \begin{subfigure}[c]{0.24\textwidth}
    \centering
    \includegraphics[width=\linewidth]
    {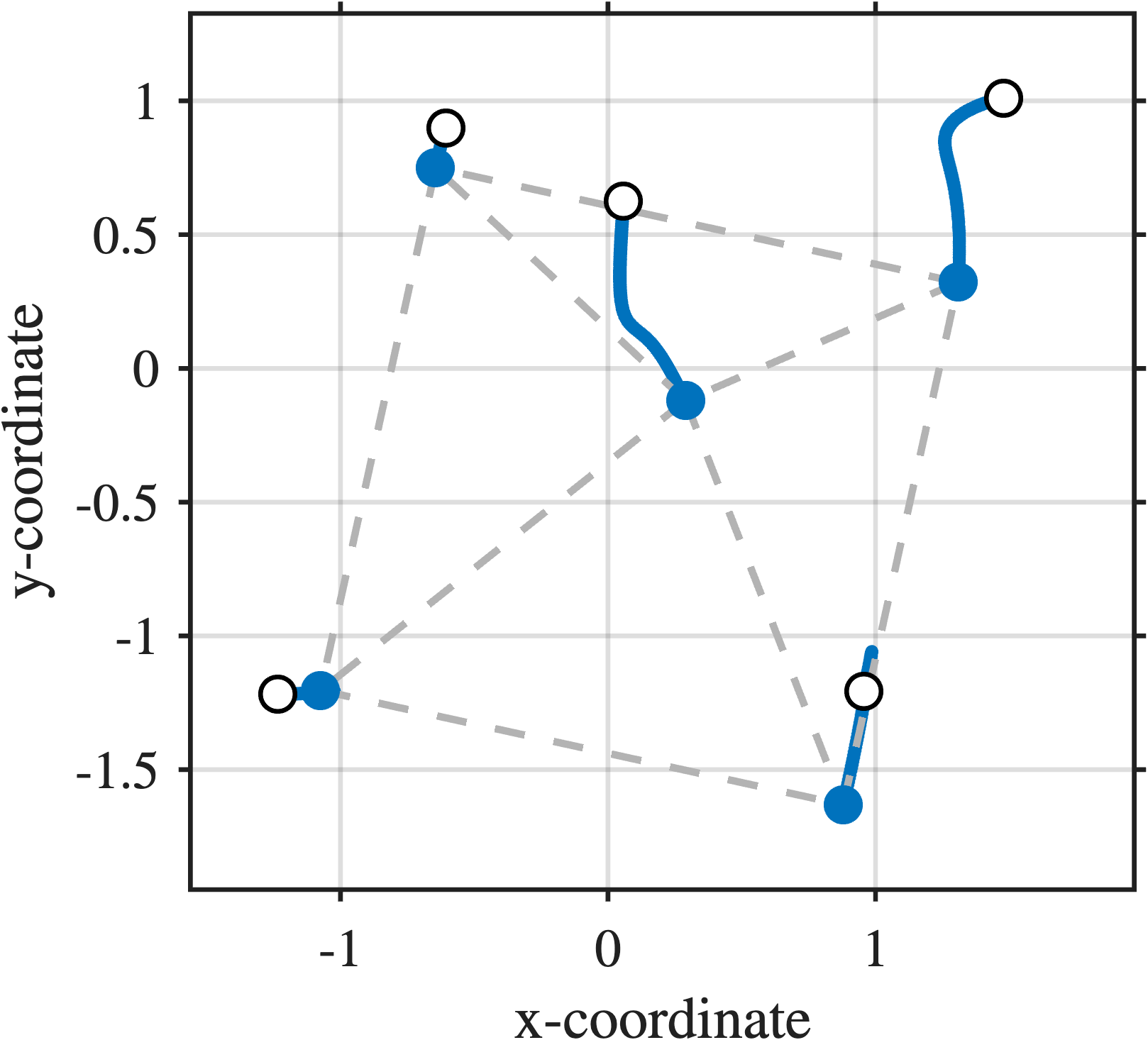}
    \caption{Node trajectories (open circles indicate the initial positions).}
    \label{fig:ex3_nodes}
  \end{subfigure}

  \caption{Directed wheel (Example~\ref{ex:wheel_hurwitz_not_pd}): $-A$ is Hurwitz
but $\frac{1}{2}(A+A^\top)$ is indefinite. The edge error converges to
zero and the formation converges to a congruent realization.
  }
  \label{fig:ex3}
\end{figure*}

\subsection{Strong admissibility and dynamics}
This next result is immediate from
Theorems \ref{thm: convergence certificates} and \ref{thm: admissible
is generic}.
\begin{corollary}[Exponential stability implies strong admissibility]
  \label{thm:admiss}
  Let Assumption~\ref{asump.drigid-geneic} hold.
  If a compatible formation controller
  (Definition~\ref{def.compatform}) is locally exponentially stable at any target formation $(p^*,m^*)$ the controller is strongly admissible.
\end{corollary}
In general, the different admissibility notions are distinct, as 
we will see in the next section.  


\section{Compatible controllers on Directed Graphs}\label{sec.directed}
An appealing feature of the gradient control \eqref{grad_Control} is
that it admits a \emph{distributed} implementation.  At the agent level,
\eqref{grad_Control} is equivalent to
\begin{equation}\label{agent_grad_control}
  u_i = \sum_{ij\in \E} (\|p_i-p_j\|^2-m^*_{ij})(p_j-p_i).
\end{equation}
Each agent therefore only requires information from its neighboring
agents defined by the graph $\G$.  This implementation, however,
assumes bidirectional sensing
as each edge contributes to the control of both incident agents.

We remove this reciprocity by orienting the edges of $\mathcal G$.
Given an orientation $\overrightarrow{\mathcal G}$,
we assign each distance constraint to its tail vertex and consider
\begin{equation}\label{dir_agent_grad_control}
  u_i = \sum_{\overrightarrow{ij}\in \E}(\|p_i-p_j\|^2-m^*_{ij})(p_j-p_i),
\end{equation}
where $\overrightarrow{ij}$ is the directed edge from node $i$ to
$j$. This model was used, for example, in \cite{Zhang_JDSMC2018}. The
aggregate dynamics of \eqref{dir_agent_grad_control}, which is no longer a
gradient flow, becomes
\begin{equation}\label{dir_formation_ctr}
  \dot p = -\overrightarrow R(p)^\top (R(p)p-m^*).
\end{equation}
The matrix $\overrightarrow{R}(p)$ is obtained from $R(p)$ by zeroing
out, in the
row corresponding to $ij$, the entries
corresponding to the node $j$ for the orientation $\overrightarrow{ij}$,
and the entries corresponding to $i$ for the orientation
$\overleftarrow{ij}$.
We henceforth define \eqref{dir_formation_ctr} to be the
\emph{($d$-dimensional) directed controller}.

The directed controller \eqref{dir_formation_ctr} naturally fits into
the compatible controller framework introduced in
Definition~\ref{def.compatform}. Indeed, from \eqref{eq:pm-dynamics},
we identify
\[
  \nu_{(p,m)}=\overrightarrow R(p)^\top.
\]
Since $\dot m=2R(p)\dot p$, the induced edge dynamics are
\[
  \dot m= 2R(p)\overrightarrow R(p)^\top (m^*-m).
\]
Thus the corresponding edge operator is
\[
  \eta_{(p,m)}= 2R(p)\overrightarrow R(p)^\top ,
\]
showing that the directed controller is compatible.
In this
section, it will be useful to have notation for the matrix of
the transformations $\eta_{(p,m)}$ and $\eta_{(p,m)}|_{T_{m} \mathcal{Q}}$.
To this end, we define
\[
  Z = 2R(p)\overrightarrow{R}(p)^\top\text{ and } A = P^\top ZP,
\]
where $P$ has orthonormal columns spanning the column space
of $R(p)$.
These are the matrices of $\eta_{(p,m)}$ in the
standard basis and $\eta_{(p,m)}|_{T_{m} \mathcal{Q}}$ in an
orthonormal basis.  The latter assertion uses that $(p,m)\in \mathcal{S}$
implies that $\IM R(p) = T_{m} \mathcal{Q}$.

\subsection{Local Stability of Directed Compatible Controllers}
For the directed controller \eqref{dir_formation_ctr}, convergence depends on
the target $(p^*,m^*)$.  Moreover, there are targets where
the certificate of convergence (a) from Theorem \ref{thm: convergence
certificates}
holds but (b) does not.  In this section, we illustrate both
phenomena with numerical examples.

\begin{example}[Hurwitz stability without the quadratic
  certificate]\label{ex:wheel_hurwitz_not_pd}
  This example shows that (a) of
  Theorem~\ref{thm: convergence certificates} does not imply (b).
  Consider the directed wheel and
  target framework shown in Figure~\ref{fig:ex3_graph}.
  A direct computation gives $\min \operatorname{Re}\lambda(A)=0.7221>0$, and    $\min_{\|v\|=1} q(v)=-0.1623$.   Thus $-A$ is Hurwitz, while the quadratic form $q$ is
  indefinite. Nevertheless, the edge error converges to zero
  (Figure~\ref{fig:ex3_edge}), and the corresponding node
  trajectories converge to a realization congruent to the target
  (Figure~\ref{fig:ex3_nodes}). Hence the quadratic-form condition is
  sufficient but not necessary for local exponential stability.
\end{example}
\begin{figure*}[!t]
  \centering
  \begin{subfigure}[c]{0.28\textwidth}
    \centering
 \includegraphics[
  width=\linewidth
]{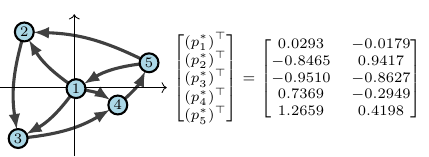}
    \caption{Directed target.}
    \label{fig:ex4_graph}
  \end{subfigure}
  \hfill
  \begin{subfigure}[c]{0.2\textwidth}
    \centering
    \includegraphics[width=\linewidth]
    {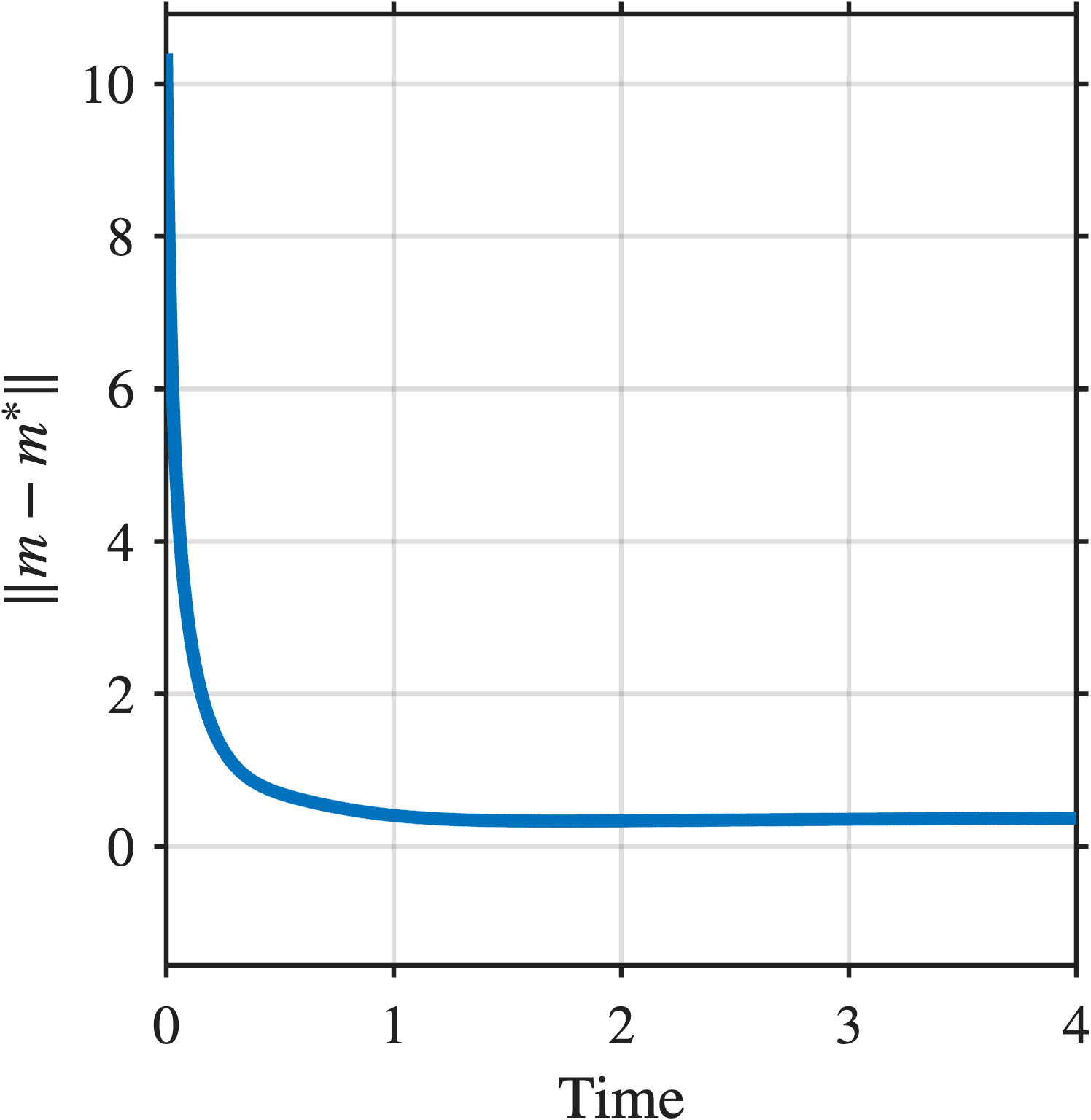}
    \caption{Edge-error norm.}
    \label{fig:ex4_edge}
  \end{subfigure}
  \hfill
  \begin{subfigure}[c]{0.2\textwidth}
    \centering
    \includegraphics[width=\linewidth]
    {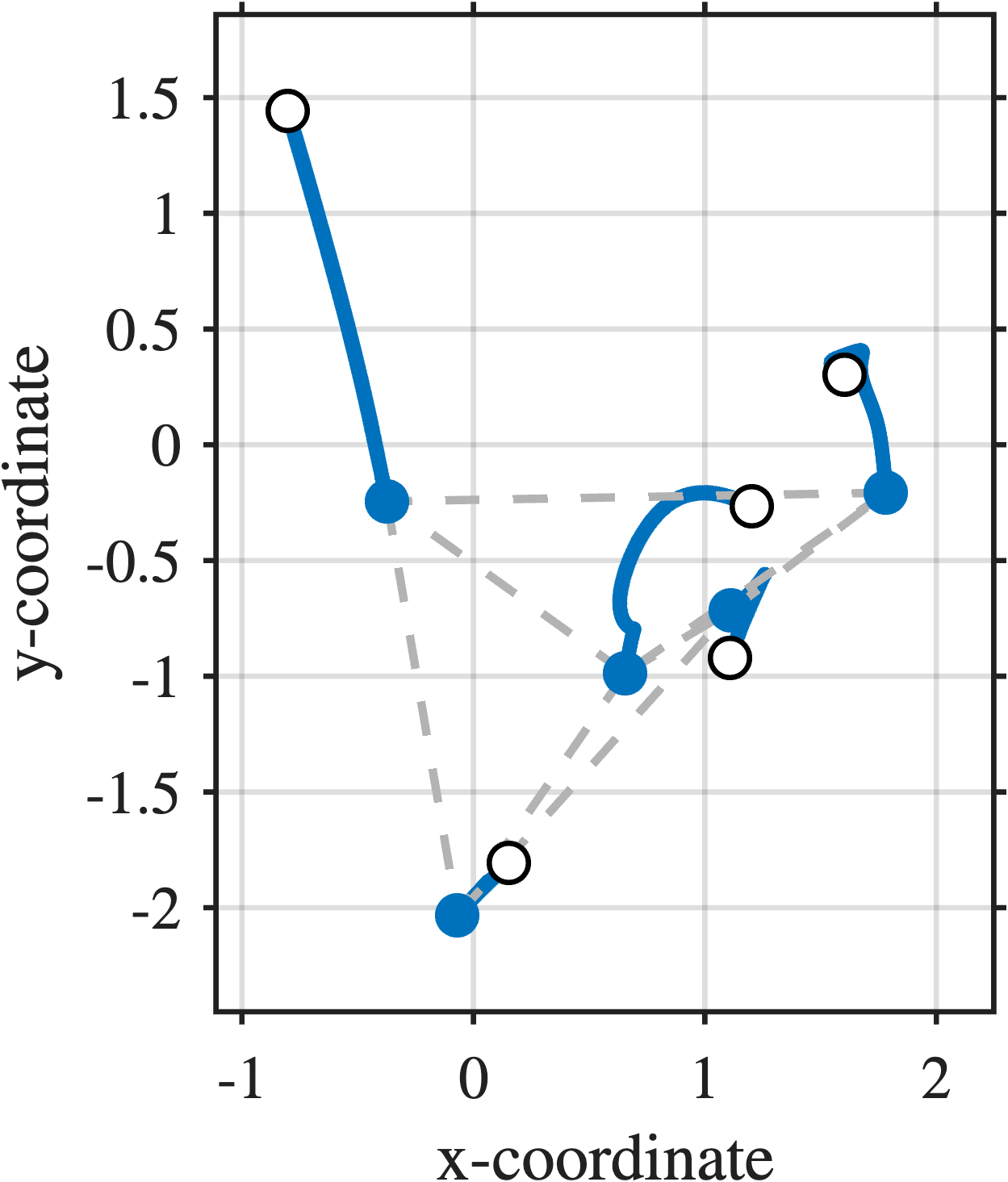}
    \caption{Node trajectories  (open circles indicate the initial positions).}
    \label{fig:ex4_nodes}
  \end{subfigure}

  \caption{
  Directed wheel (Example~\ref{ex:wheel_not_hurwitz_not_pd}): $-A$ is not Hurwitz. The edge error settles at a nonzero value and the formation converges to a non-congruent configuration.
  }
  \label{fig:ex4}
\end{figure*}
\begin{example}[Failure of both stability tests]
  \label{ex:wheel_not_hurwitz_not_pd}
  Consider the directed wheel and target framework shown in
  Figure~\ref{fig:ex4_graph}.
  A direct computation gives $\min \operatorname{Re}\lambda(A)=-0.0598<0$, and    $min_{\|v\|=1} q(v)=-0.2012$.  Thus $-A$ is not Hurwitz and the quadratic form $q$ is
  indefinite. The edge error converges to a nonzero value
  (Figure~\ref{fig:ex4_edge}), and the corresponding node
  trajectories converge to a configuration that is not congruent
  to the target (Figure~\ref{fig:ex4_nodes}). Together with
  Example~\ref{ex:wheel_hurwitz_not_pd}, this example shows
  that local stability can depend on the target geometry even when the
  directed orientation is fixed.
  %
\end{example}

Theorem~\ref{thm: convergence certificates} therefore gives a
complete stability test for a fixed orientation
and target geometry. To obtain information about the orientation
independently of a particular target,
we now return to the generic admissibility notions
introduced in Section \ref{sec.admissibility}.
\subsection{Admissibility for Directed Orientations}

For the directed controller, generic  admissibility defines a property of the
orientation $\overrightarrow{\mathcal G}$.
In this case,  we refer to $\overrightarrow{\G}$ as a
\emph{(weakly/strongly) admissible orientation} of  $\G$. Unlike the
symmetric controllers of Section \ref{sec.geoframework},
the three notions of Definition~\ref{def: admissible} need not
coincide for the directed controllers,
which we show with the following example.

\begin{example}[Weak admissibility does not imply admissibility]
  \label{ex:weak-not-admissible}
  Consider the directed graph in Figure~\ref{fig:weak-not-admissible}.
  A direct computation gives $\rk Z=9=2n-3$.
  Since \(9\) is the maximal possible rank and the entries of \(Z\)
  depend polynomially on the target
  configuration, it follows that $\rk Z=9$ for every generic target.
  Hence the controller is
  weakly admissible. To check admissibility, we compute the eigenvalues
  of the reduced
  matrix \(P^\top ZP\) as $\{29.6147\pm2.8990
  i,\;22.1401,\;14.8599,\;10.0000,\;7.7485\pm0.1488 i,\;3.2737,\;0\}$.
  The zero eigenvalue  shows that the reduced matrix is not invertible.
  Therefore, the controller is weakly admissible but not admissible.
  \begin{figure}[!ht]
    \centering
\includegraphics[
  width=.5\columnwidth
]{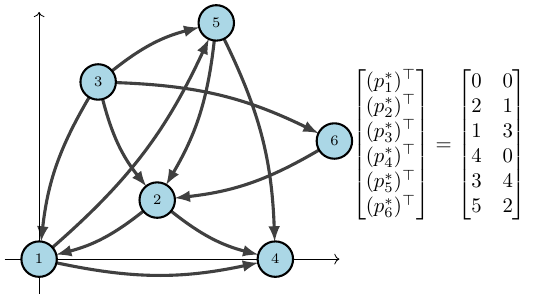}
    \caption{A directed formation that is weakly admissible but not
      admissible. }
    \label{fig:weak-not-admissible}
  \end{figure}
\end{example}


We conclude this section by showing that every $d$-rigid graph has an
admissible orientation.
\begin{theorem}[Admissible orientations exist]\label{thm: admissible
  orientations}
  Let $\G$ be a generically $d$-rigid graph.  Then there is an
  admissible orientation $\overrightarrow{\G}$ of $\G$.
\end{theorem}
In the proof, we need a simple vanishing result for multiaffine
polynomials.
\begin{lemma}\label{lem: multiaffine}
  Let $A({\bf t})$ be a multiaffine polynomial\footnote{A polynomial
    in ${\bf t}=(t_1,\ldots,t_m)$ is \emph{multiaffine} if it has degree
  at most one in each variable individually.}
  in variables $t_1, \ldots, t_m$.  Then
  $A$ vanishes identically if and only if it vanishes on all zero/one
  vectors ${\bf x}\in \{0,1\}^m$.
\end{lemma}
The proof of the Lemma is given in Appendix \ref{pf_lemma3}.
We also need a specialized sufficient condition for a
matrix with polynomial entries to have a multiaffine determinant.
\begin{lemma}\label{lem: weighted determinant}
  Let $A({\bf t})\in \R[{\bf t}]^{m\times m}$ and $B\in \R^{m\times p}$ be
  matrices with $p \le m$.\footnote{Here, $\R[{\bf t}]$, with ${\bf
    t}=(t_1,\ldots,t_k)$,
    denotes the set of polynomials in $t_1,\ldots,t_k$ with real
    coefficients. Thus $A({\bf t})$
  is a matrix whose entries are polynomials in ${\bf t}$.}  If every
  $p\times p$ minor of $A({\bf t})$
  is multiaffine then $\det B^\top A({\bf t})B$ is multiaffine.
\end{lemma}
The proof of the Lemma is given in Appendix \ref{pf_lemma4}.

\begin{proof}[Proof of Theorem \ref{thm: admissible orientations}]
  Let $\G$ be generically $d$-rigid, and fix a generic configuration
  $p^*$.   Let $P$ be a matrix with orthonormal columns spanning
  the column space of $R$.

  For each edge $ij\in \E$,
  define $r_{ij}$ to be the row of the
  rigidity matrix corresponding to the edge $ij$ and $\overrightarrow{r_{ij}}$
  the vector obtained by zeroing out the entries corresponding to $j$ in
  $2r_{ij}$; define $\overleftarrow{r_{ij}}$ similarly so that
  $r_{ij} = \frac{1}{2}\overrightarrow{r_{ij}} +
  \frac{1}{2}\overleftarrow{r_{ij}}$ (which is the reason 
  we scaled the directed rows in this proof).
  Denote by $\overrightarrow{R}$ the matrix with rows
  $\overrightarrow{r_{ij}}$ and
  $\overleftarrow{R}$ the matrix with rows $\overleftarrow{r_{ij}}$,
  ordered compatibly,
  so that $R = \frac{1}{2}\overrightarrow{R} + \frac{1}{2}\overleftarrow{R}$.

  Now we define a matrix
  $Z({\bf t}) = R[T\overrightarrow{R} + (I - T)\overleftarrow{R}]^\top$, where
  $T$ is an $|\E|\times |\E|$ diagonal matrix with variables $t_{ij}$ on the
  diagonal, and then $A({\bf t}) = P^\top Z({\bf t})P$.  For each ${\bf
  x}\in\{0,1\}^{|\E|}$, $Z({\bf x})$ corresponds
  to the orientation $\overrightarrow{\G}_{\bf x}$ obtained by
  orienting $ij$ as $\overrightarrow{ij}$ if $x_{ij}=1$ and as
  $\overleftarrow{ij}$ if $x_{ij}=0$. Thus, it suffices to find ${\bf
  x}\in\{0,1\}^{|\E|}$ such that
  $\det A({\bf x})\neq 0$, since this implies that
  $\eta^*|_{T_{m^*}\mathcal Q}$ is invertible for
  $\overrightarrow{\G}_{\bf x}$.


  We first show that $\det A({\bf t})$ is multi-affine. This follows
  from Lemma \ref{lem: weighted determinant} once we observe that
  each variable $t_{ij}$ appears in at most one column of
  $Z({\bf t})$ and that $P$ is a fixed constant.  Because
  $A(\frac{1}{2}{\bf 1}) = P^\top RR^\top P$ is positive definite,
  $\det A(\frac{1}{2}{\bf 1}) > 0$.  Lemma
  \ref{lem: multiaffine} then implies that there is some
  ${\bf x}\in \{0,1\}^{\E}$ such that $\det A({\bf x}) \neq 0$.
  This shows that the directed orientation $\overrightarrow{\G}_{\bf x}$
  is admissible at the target $p^*$.  That $\overrightarrow{G}_{\bf x}$
  is generically admissible now follows from
  Theorem \ref{thm: admissible is generic}.
\end{proof}

\subsection{Acyclic orientations and the vertex flag}
Acyclic orientations play an important role in the literature on
directed formation control.  The edge-based perspective in this
paper allows us to give a simple explanation of why this is
so and unify much of what is known.
\begin{theorem}[Stability of acyclic controllers]\label{thm: acyclic}
  If $\overrightarrow{\G}$ is acyclic, then the following are equivalent:
  \begin{enumerate}[(a)]
    \item $\overrightarrow{\G}$ is weakly admissible;
    \item $\overrightarrow{\G}$ is admissible;
    \item $\overrightarrow{\G}$ is strongly admissible;
    \item the map $-\eta^*|_{T_{m^*} \mathcal{Q}}$ is Hurwitz
      at every generic target $(p^*,m^*)$.
  \end{enumerate}
\end{theorem}
The proof is an application of Theorem \ref{thm: flag} to a
flag naturally associated with the orientation that we now describe.
Suppose that $\overrightarrow{\G}$
is an acyclic directed graph, and that the vertices are topologically
ordered so that the
edge $ij$ has orientation $\overrightarrow{ij}$ if and only if $i < j$; such
an ordering exists if and only if $\overrightarrow{G}$ is acyclic.
We define the
following subspaces of $\R^{|\E|}$: $V_0 = \{0\}$,
and, for $1\le j \le |\V|$, $V_j$ is the coordinate subspace spanned by
coordinates corresponding to edges $\overrightarrow{ij}$ with their tail in the
initial segment $1\le i\le j$.  This implies that $V_{|\V|} = \R^{|\E|}$.
Plainly the $V_i$ form a flag of subspaces which we call the \emph{vertex flag}
of $\overrightarrow{\G}$.
\begin{lemma}\label{lem: vertex flag}
  Let $\overrightarrow{\G}$ be an acyclic directed graph and let $V_j$
  be the vertex flag.  Then, for all
  targets $p^*\in \mathcal{C}$ of the associated directed controller
  with $Z = 2R(p^*)\overrightarrow{R}(p^*)^\top$,
  \begin{enumerate}[(a)]
    \item each $V_j$ is $Z$-invariant;
    \item the induced maps $Z_j: W_j \to W_j$ along the quotients $W_j =
      V_j/V_{j-1}$
      are PSD and satisfy \eqref{eq: skew};
    \item the vertex flag satisfies the decoupling condition.
  \end{enumerate}
\end{lemma}
\begin{proof}
  We require some additional definitions. Define $X_j$ to be the
  coordinate subspace of
  $\R^{|\E|}$ corresponding to the edges out of vertex $j$. Then $V_j =
  V_{j-1}\oplus X_j$.
  Let $\pi_j:V_j\to W_j:=V_j/V_{j-1}$ be the quotient map. Since
  $\pi_j(X_j)=W_j$ and
  $X_j\cap\ker\pi_j=X_j\cap V_{j-1}=\{0\}$, the restriction
  $\pi_j|_{X_j}:X_j\to W_j$ is
  injective. Since $\dim X_j = \dim V_j - \dim V_{j-1} = \dim W_j$,
  $\pi_j|_{X_j}$ is a linear isomorphism. We also define $U_j$ to be
  the coordinate
  subspace of $\R^{dn}$ corresponding to the vertices $i$ such that
  $1\le i\le j$. Finally, we
  define $Q_j$ to be the matrix whose rows are the edge vectors
  emanating from vertex $j$.

  We first prove (b). Note that
  on $X_j$, $\overrightarrow{R}(p)^\top$ acts like $Q_j^\top$.
  On $U_j$, $R(p)$ acts like $Q_j$. From the previous observation, $Z_j
  = 2Q_jQ_j^\top$,
  which is PSD.  In particular, $Z_j$ satisfies \eqref{eq: skew}.

  Now we prove (a). Since $V_j$ is supported on edges whose tails lie in
  $\{1,\ldots,j\}$,
  and each row of $\overrightarrow{R}(p)$ is supported only on the
  coordinates of its tail vertex, we have
  $\overrightarrow{R}(p)^\top V_j\subseteq U_j$. Conversely, if $u\in
  U_j$, then $(R(p)u)_{kl}$
  can be nonzero only if at least one of $k$ or $l$ is at most $j$.
  Since the vertices are topologically ordered, the
  tail of such an edge is the endpoint with smaller index, and hence is
  also at most $j$.
  Thus $R(p)U_j\subseteq V_j$. It follows that
  $ZV_j=2R(p)\overrightarrow{R}(p)^\top V_j\subseteq V_j$, so $V_j$ is
  $Z$-invariant.

  We finish with the decoupling condition (c) using criterion (b) from
  Lemma \ref{lem: decoupling}.
  Let $v\in V_j$ and suppose that
  $Zv \in V_{j-1}$.  Using the splitting of $V_j$, we write $v = w + x$,
  with $w\in V_{j-1}$
  and $x\in X_j$.  We have $0 = \pi_j(Zv) = \pi_j(Zw) + \pi_j(Zx)$.  By
  invariance, $Zw\in V_{j-1}$, so $ \pi_j(Zw)$ vanishes.  This leaves
  $0 = \pi_j(Zx) = Z_j\pi_j(x)$.  Since the kernel of the PSD $Z_j$ is
  equal to $\ker Q_j^\top$, $\pi_j(x)\in \ker Q_j^\top$.  On $X_j$,
  $\overrightarrow{R}(p)^\top$ acts like $Q_j^\top$, and so
  $\overrightarrow{R}(p)^\top x = 0$. Hence $Zx = 0$, and so $Zv = Zw$.
  This verifies the decoupling condition, completing the proof.
\end{proof}
\begin{proof}[Proof of Theorem \ref{thm: acyclic}]
  Let $p^*$ be a generic target, and assume that the
  directed controller from $\overrightarrow{\G}$ is weakly
  admissible.  Because it is a constant, the vertex flag
  depends rationally on the node state.  The remaining
  hypotheses of
  Theorem \ref{thm: flag} are supplied by Lemma \ref{lem: vertex flag},
  since they hold pointwise, and so at $p^*$, in particular.
  Because the spectrum of $\eta^*$ is the union of the
  spectra of the PSD $Z_j$ by Lemma \ref{lem: flag}, the
  non-zero eigenvalues of $\eta^*$ have positive real parts.
  Hence, $-\eta^*|_{T_{m^*} \mathcal{Q}}$ is Hurwitz by Theorem \ref{thm: flag}.
  Since $p^*$ was arbitrary, (a) implies (d) for every generic $p^*$.  The
  implications from bottom to top are immediate for any generic $p^*$,
  so the proof is complete.
\end{proof}

We can apply the same machinery to other invariant flags associated
with a directed orientation.  The following construction generalizes
a related approach from \cite{Zhang_JDSMC2018}.
\begin{definition}[Root extension]
  Let $\overrightarrow{\G}$ be a directed graph and let
  $\overrightarrow{\mathcal H}\subseteq\overrightarrow{\G}$ be a
  vertex-induced subgraph.  We say that $\overrightarrow{\G}$ is a
  \emph{root extension} of $\overrightarrow{\mathcal H}$ if
  \begin{enumerate}[(a)]
    \item $\overrightarrow{\G}\setminus\overrightarrow{\mathcal H}$
      is acyclic; and
    \item there are no directed edges from a vertex of
      $\overrightarrow{\mathcal H}$ to a vertex of
      $\overrightarrow{\G}\setminus\overrightarrow{\mathcal H}$.
  \end{enumerate}
  If, in addition, $\overrightarrow{\mathcal H}$ is admissible and every
  vertex of $\overrightarrow{\G}\setminus \overrightarrow{\mathcal H}$
  has out-degree at least $d$, then
  $\overrightarrow{\G}$ is a \emph{stable root extension} of
  $\overrightarrow{\mathcal H}$.
\end{definition}
The construction is illustrated in Fig.~\ref{fig:acyclic_extension}.
\begin{figure}[!ht]
  \centering
\includegraphics[
  width=0.4\columnwidth
]{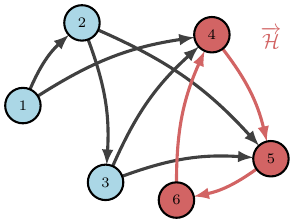}
  \caption{A stable root extension by the admissible directed subgraph
    $\overrightarrow{\mathcal H}$ (red).  The subgraph
    $\overrightarrow{\G}\setminus\overrightarrow{\mathcal H}$ is
    acyclic, all edges between the two subgraphs are directed
    toward $\overrightarrow{\mathcal H}$, and the vertices outside of
    $\overrightarrow{\G}\setminus\overrightarrow{\mathcal H}$
  have out-degree at least $d = 2$.}
  \label{fig:acyclic_extension}
\end{figure}
\begin{figure*}[t]
  \centering

  \begin{subfigure}[c]{0.28\textwidth}
    \centering
  \includegraphics[
  width=0.90\linewidth
]{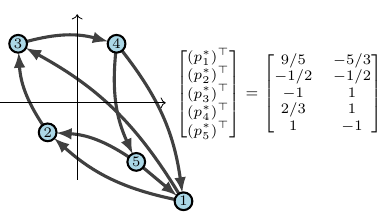}
    \caption{Persistent directed target.}
    \label{fig:persistent_bad_graph}
  \end{subfigure}
  \hfill
  \begin{subfigure}[c]{0.2\textwidth}
    \centering
    \includegraphics[width=\linewidth]
    {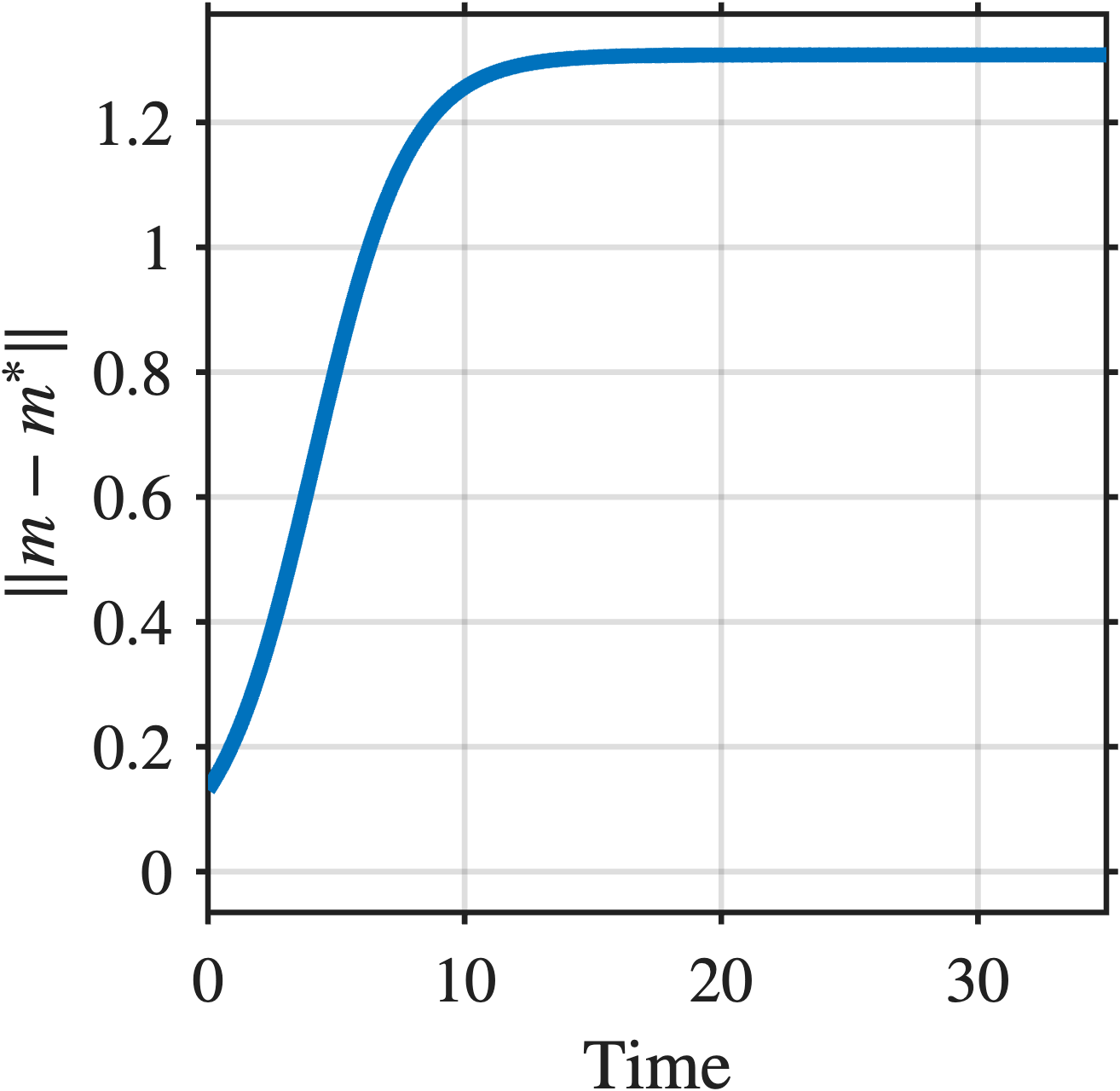}
    \caption{Edge-error norm.}
    \label{fig:persistent_bad_edge}
  \end{subfigure}
  \hfill
  \begin{subfigure}[c]{0.24\textwidth}
    \centering
    \includegraphics[width=\linewidth]
    {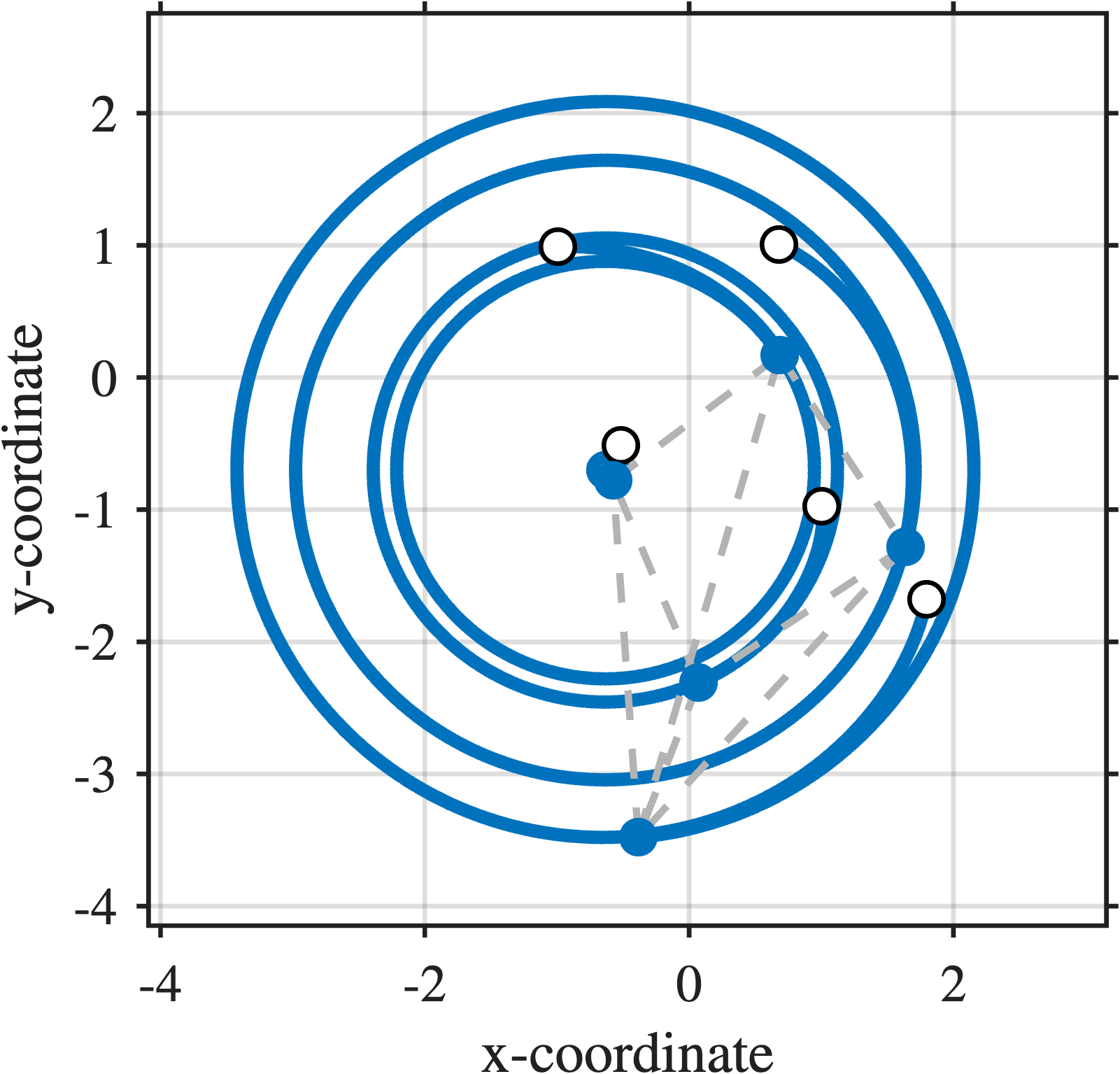}
    \caption{Node trajectories  (open circles indicate the initial positions).}
    \label{fig:persistent_bad_nodes}
  \end{subfigure}

  \caption{A persistent directed formation for which the directed
    controller is locally unstable. The orientation in
    (a) is persistent. The edge error (b) converges to a nonzero value,
    while the node
    trajectories (c) approach a limit cycle. }
  \label{fig:persistent_not_sufficient}
\end{figure*}

The definition of a root extension implies that there
is an ordering $v_1, \ldots, v_n, w_1, \ldots, w_\ell$ of the vertices
of $\overrightarrow{\G}$,
where the $v_i$ correspond to vertices of
$\overrightarrow{\G}\setminus  \overrightarrow{\mathcal{H}}$
and the $w_j$ correspond to vertices of $\overrightarrow{\mathcal{H}}$,
such that for each edge $\overrightarrow{v_iv_j}$, $i < j$ and there
are no directed
edges $\overrightarrow{w_iv_j}$.  Using such an order, which extends
a topological
order on the vertices of $\overrightarrow{\G'}$, we define a flag $V_j$ of
subspaces of $\R^{|\E|}$ by defining $V_j$ to be the coordinate
subspace corresponding
to edges with their tail in $\{v_1, \ldots, v_j\}$ for $0\le j\le n$ and
$V_{n+1} = \R^{|\E|}$.  We call this flag $V_j$ the \emph{stable
extension flag}.
\begin{theorem}[Stability of root extensions]\label{thm: stable root extension}
  Let $\overrightarrow{\G}$ be a $k$ vertex stable root extension of
  an $\ell$ vertex admissible directed graph $\overrightarrow{\mathcal{H}}$.
  Then $\overrightarrow{\G}$ is admissible.
  If for some generic target $p^*$, the non-zero eigenvalues of
  $Z_{\overrightarrow{{\mathcal{H}}}}$
  have positive real parts, where $Z_{\overrightarrow{\mathcal{H}}}$ is
  the matrix of the
  edge operator at the restriction of $p^*$ to $\overrightarrow{\mathcal{H}}$,
  then $\overrightarrow{\G}$ is strongly admissible.
\end{theorem}
\begin{proof}
  Let $D = \binom{d+1}{2}$.  Let $n = k - \ell$, and order the 
  vertices of $\overrightarrow{\G}$ as 
  $v_1, \ldots, v_n, w_1, \ldots, w_\ell$ with the conditions 
  as discussed above. 
  With an eye towards applying Theorem \ref{thm: flag}, we first observe that,
  as a constant flag, the stable extension flag depends rationally
  on the state.  It is also a flag of $Z$ invariant subspaces, because it
  agrees with the vertex flag up to $V_n$, so the proof of invariance
  applies verbatim, except for $V_{n+1}$, where invariance is trivial.

  Next we look at the induced maps $Z_j$.   For $1\le j\le n$,
  $Z_j = 2Q_jQ_j^\top$, as in the proof of Lemma \ref{lem: vertex flag},
  with the same argument.
  Since  every vertex outside $\overrightarrow{\mathcal{H}}$ has out-degree at least
  $d$, generically, $\rk Z_j = d$ for $1\le j\le n$.
  We also get $Z_{n+1} = Z_{\overrightarrow{\mathcal{H}}}$,
  since, by construction, $\overrightarrow{\mathcal{H}}$ has no out edges.
  By admissibility and Theorem \ref{thm: weak to strong},
  $Z_{\overrightarrow{\mathcal{H}}}$ satisfies \eqref{eq: skew}.

  From the general rank inequality  $\rk Z \ge \rk Z_1 + \cdots + \rk Z_{n+1} =
  dn + \rk Z_{\overrightarrow{\mathcal{H}}}$.  Since
  $\overrightarrow{\mathcal{H}}$
  has $k - n$ vertices and $\overrightarrow{\mathcal{H}}$ is admissible, generically,
  $\rk Z \ge dk - D$.  Since $\IM Z \subseteq T_{m^*}\mathcal{Q}$, which
  has dimension at most $dk - D$, we conclude both that $\overrightarrow{\G}$
  is $d$-rigid and that equality holds.  This shows that
  $\overrightarrow{\G}$ is
  weakly admissible, and, by Lemma \ref{lem: decoupling} (d) that the
  decoupling condition holds.

  At this point, we have established all the hypotheses of Theorem
  \ref{thm: flag},
  so we  conclude that $\overrightarrow{\G}$ is admissible.  If, at some
  generic target the non-zero eigenvalues
  of  $Z_{\overrightarrow{\mathcal{H}}}$ have positive real parts, the same is
  true for $Z$ because the other $Z_j$ are PSD, using Lemma \ref{lem:
  flag}.  The
  strong admissibility conclusion then follows.
\end{proof}
As a last application of vertex flags, we prove a combinatorial
theorem about weakly admissible acyclic orientations.
We define the \emph{defect}  of vertex $i$ as $\delta_i:=\max\{0,d - \overrightarrow{\deg}(i)\}$,
and the \emph{defect of an acyclic graph} to be the sum of its vertex defects.

\begin{theorem}[Admissible acyclic orientations characterization]\label{thm: weak admiss acyclic}
  Let $\overrightarrow{\G}$ be an acyclic
  orientation of an $n\ge d$ vertex $d$-rigid graph $\G$.
  Then the following are equivalent:
  \begin{enumerate}[(a)]
    \item $\overrightarrow{\mathcal{G}}$ is admissible;
    \item the defect of $\overrightarrow{\mathcal{G}}$ is exactly $\binom{d+1}{2}$;
    \item $\overrightarrow{\mathcal{G}}$ is formed from a directed complete $d$ vertex subgraph $\overrightarrow{\mathcal{H}}$ by a sequence of stable root extensions.
  \end{enumerate}
\end{theorem}
\begin{proof}
  For convenience, set $D = \binom{d+1}{2}$.
  Fix any topological ordering $v_1,\ldots,v_n$ of the vertices of 
  $\overrightarrow{\mathcal{G}}$, and let $V_1,\ldots,V_n$ be the associated vertex flag.
  In a topological ordering, the out-neighbors of any vertex appear after it.
  In particular, each vertex $v_{n-k}$ for $0\le k\le d-1$ has at most
  $k$ out neighbors. 
  Summing over $k$, the total vertex defect in
  any acyclic orientation is at least $D$.  Moreover, defect
  $D$ is attained if and only if each of these $v_{n-k}$ for $0\le k\le d-1$ 
  have out-degree exactly $k$ and all other defects are zero.
  From this it follows that conditions (b) and (c) are equivalent.

  Meanwhile, by Lemma \ref{lem: vertex flag} the vertex
  flag satisfies the decoupling condition.  Hence,
  from Lemma \ref{lem: decoupling} (d),
  generically $\rk Z = \rk Z_1 + \cdots + \rk Z_n$.
  In the notation of the proof of Lemma
  \ref{lem: vertex flag}, $Z_j = 2Q_jQ_j^\top$, which
  generically has rank $d -\delta_i$. Since each $Z_j$ is positive semidefinite, 
  condition~\eqref{eq: skew} holds for each induced map.  
  Hence, $\overrightarrow{\G}$ is weakly admissible, and
  hence, from Theorem \ref{thm: flag} admissible, if and only if
  the total defect is $D$.
\end{proof}

Theorems \ref{thm: acyclic} and \ref{thm: weak admiss acyclic} together provide an explanation for the ``leader first-follower'' directed control 
architecture appearing in works such as \cite{Anderson_CDC2007} and 
\cite{Yu_SIAMJCO2009}.  
Theorem \ref{thm: weak admiss acyclic} shows that
the defect is exactly $D$ if and only if the directed graph has a 
(generalized) leader first-follower structure.  We note that, 
in dimensions $d\ge 2$ there are generically rigid bipartite 
graphs.  By Theorem \ref{thm: weak admiss acyclic}, 
these cannot have acyclic admissible orientations, so all the 
orientations supplied by Theorem \ref{thm: admissible orientations}
must contain directed cycles.


\subsection{Comparison to persistence}

\begin{figure*}[!t]
  \centering

  \begin{subfigure}[c]{0.28\textwidth}
    \centering
   \includegraphics[
  width=0.90\linewidth
]{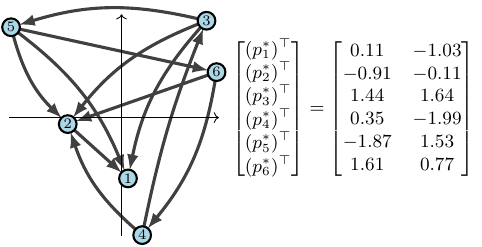}
    \caption{Directed target.}
    \label{fig:ex7_graph}
  \end{subfigure}
  \hfill
  \begin{subfigure}[c]{0.18\textwidth}
    \centering
    \includegraphics[width=\linewidth]
    {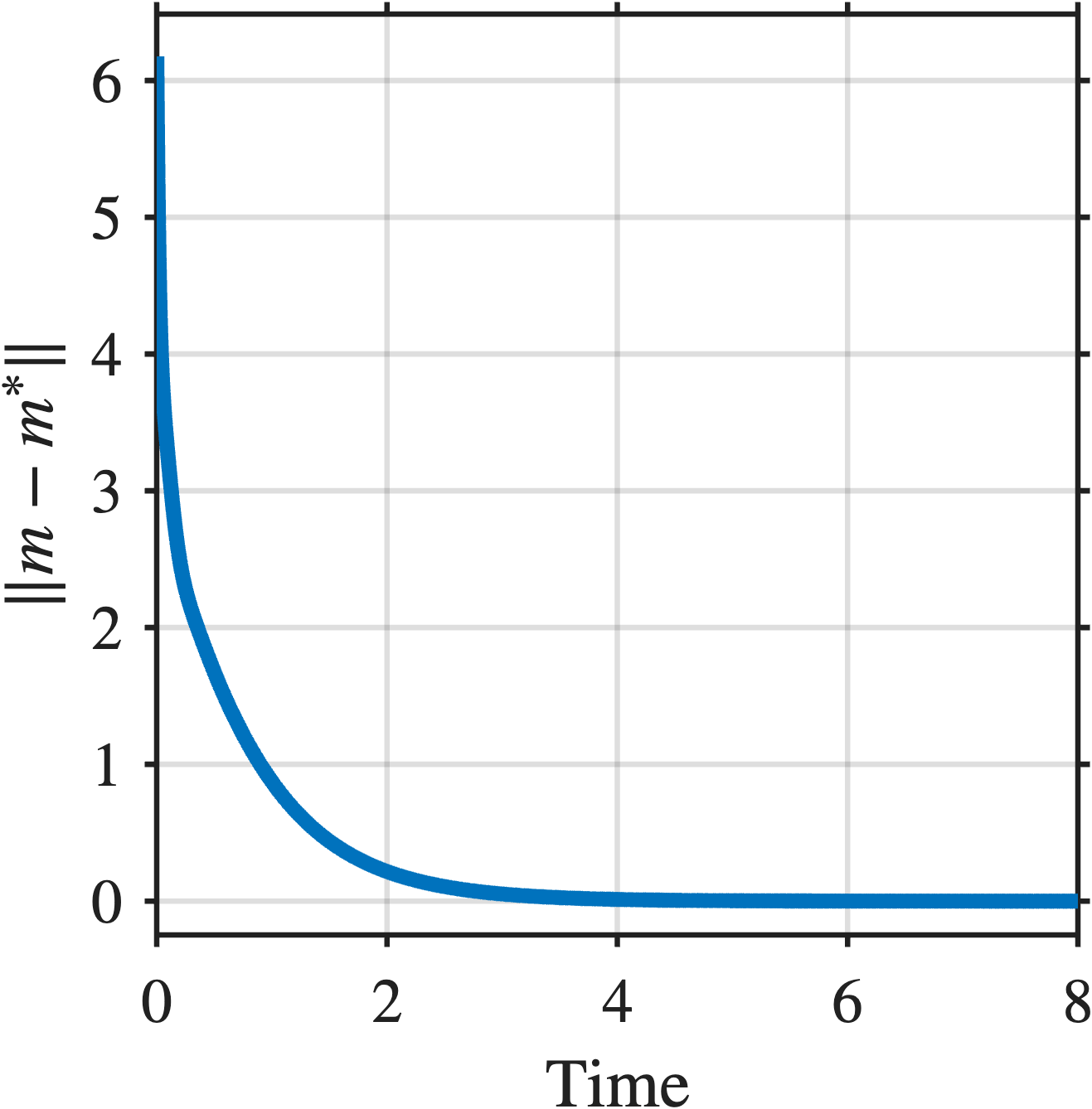}
    \caption{Edge-error norm.}
    \label{fig:ex7_edge}
  \end{subfigure}
  \hfill
  \begin{subfigure}[c]{0.22\textwidth}
    \centering
    \includegraphics[width=\linewidth]
    {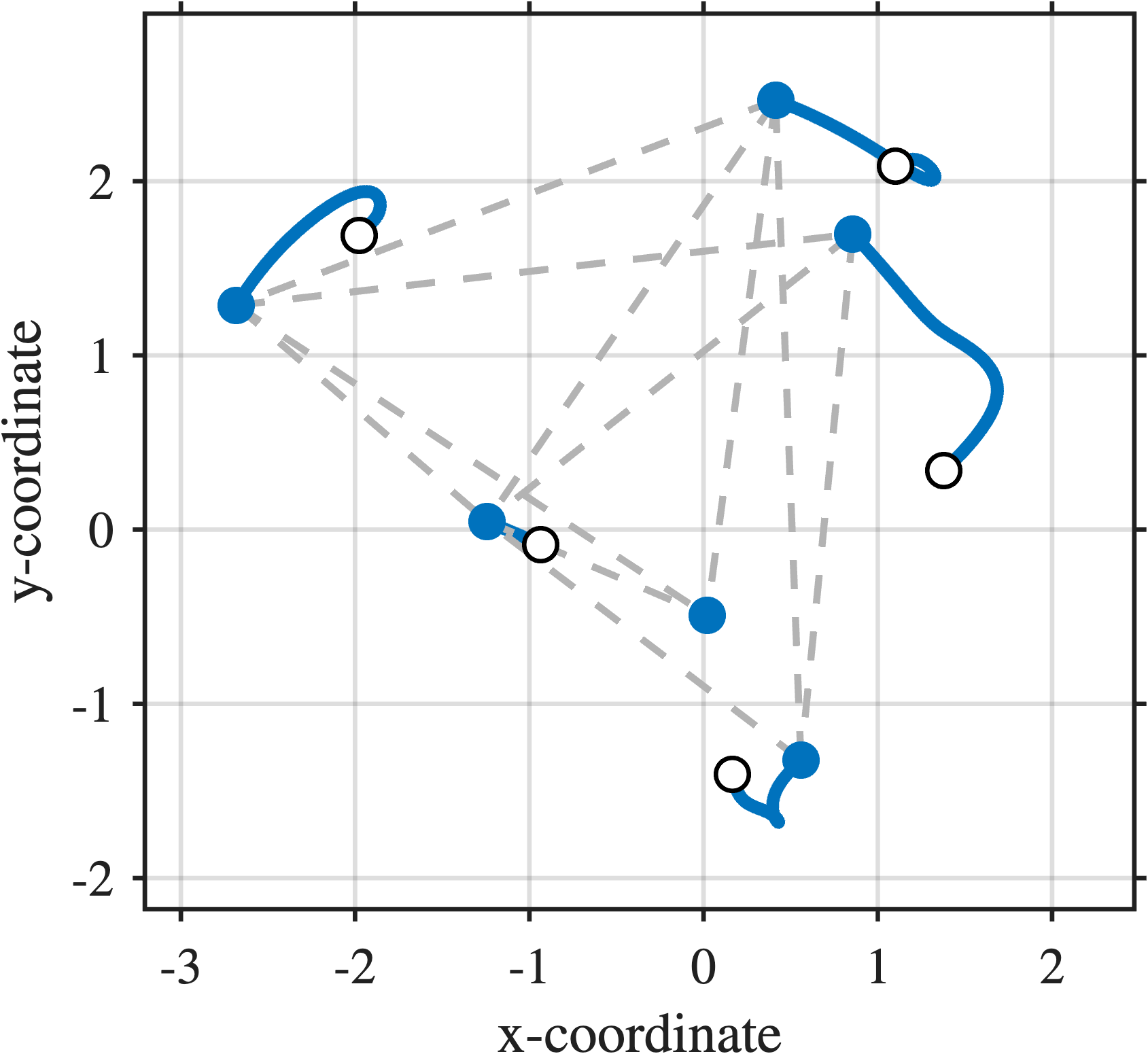}
    \caption{Node trajectories  (open circles indicate the initial positions).}
    \label{fig:ex7_nodes}
  \end{subfigure}

  \caption{
  Non-persistent orientation of Example~\ref{ex:persistence-not-necessary}.
The edge error converges to zero and the node trajectories converge to
a congruent realization.
  }
  \label{fig:nonpersistent_convergence}
\end{figure*}

We first compare admissibility with the classical notion of persistence \cite{Hendrickx_IJRNC2007}. 
For acyclic orientations, the two notions coincide.
\begin{corollary}[Acyclic persistent orientations]\label{cor: admiss equiv persist}
   An acyclic directed graph is admissible if and only if it is persistent.
\end{corollary}
\begin{proof}
    Suppose that $\G$ is admissible and acyclic. By Theorem \ref{thm: weak admiss acyclic}, 
    $\G$ contains a $K_d$ subgraph carrying the total defect $D$ and every other vertex 
    has degree at least $d$.  Removing edges with their tail at a vertex of degree 
    at least $d+1$ preserves this property, so removing edges until all degrees are at 
    most $d$ preserves admissibility (using Theorem \ref{thm: weak admiss acyclic}) 
    and, hence, rigidity of the underlying graph.  This shows that $\G$ is persistent.

    Now suppose that $\G$ is persistent. Remove edges until every vertex has degree 
    at most $d$.  The graph that is left is minimally rigid, so, by counts, it has defect 
    exactly $D$.  In an acyclic orientation, this is the minimal defect, so the defect of 
    $\G$ is exactly $D$.  By Theorem \ref{thm: weak admiss acyclic}, $\G$ is admissible.
\end{proof}

For $d=2$, Corollary \ref{cor: admiss equiv persist}, together with Theorem \ref{thm: weak admiss acyclic}, recovers the classical characterization of persistent acyclic graphs: there is a leader of out-degree zero, a first follower of out-degree one, and every remaining vertex has out-degree at least two \cite[Thm.~5]{Hendrickx_IJRNC2007}. 

The numerical characterization developed thus far provides a
target-dependent test for local convergence that is independent of
the notion of persistence \cite{Hendrickx_IJRNC2007}. Persistence is
a structural property of a
directed formation, as it characterizes whether the directed distance
constraints are compatible with
maintaining a rigid shape. It does not, however, characterize the
stability of a particular
feedback law. Indeed, previous stability results for persistent
formations have required either
controller design or additional restrictions on the directed
architecture \cite{Yu_SIAMJCO2009,Zhang_JDSMC2018}.

\begin{theorem}\label{thm:persistence}
  Persistence is neither necessary nor sufficient for local exponential
  stability of the directed  controller \eqref{dir_formation_ctr}.
\end{theorem}

The proof is given through the construction of two counter examples given below.
\begin{example}[Persistence is not sufficient]
  \label{ex:persistence-not-sufficient}
  Consider the persistent directed wheel and target framework shown in
  Figure~\ref{fig:persistent_bad_graph}.
  The underlying undirected wheel graph is generically rigid in dimension $2$, 
  and the out-degrees are $(2,1,1,2,2)$, so the orientation is persistent. For
  this target, computation gives   
  $ \min \operatorname{Re}\lambda(A)\approx -0.522<0.$
  Hence $-A$ is not Hurwitz, and by
  Theorem~\ref{thm: convergence certificates} the target is not locally
  exponentially stable under the directed controller. 
  Figures~\ref{fig:persistent_bad_edge}-\ref{fig:persistent_bad_nodes} illustrate 
  the resulting local instability using an initial perturbation along 
  an unstable mode of the linearization.  
\end{example}
We note that the orientation in Figure~\ref{fig:persistent_bad_graph} is 
strongly admissible (certified by the given configuration), 
and that, for an open set of targets (not shown), $-A$ is Hurwitz.
The second example shows that 
local exponential stability at a generic target does not imply 
persistence. 
\begin{example}[Persistence is not necessary]
  \label{ex:persistence-not-necessary}
  Consider the non-persistent directed graph and target framework shown
  in Figure~\ref{fig:ex7_graph}.
  The underlying undirected graph is rigid, but the orientation is not
  persistent. Indeed, vertices $3$ and $5$ both have out-degree three,
  while the graph has $10=2n-2$ edges. Any minimally persistent
  spanning subgraph would therefore be obtained by removing a single
  edge and, in the plane, would require every vertex to have out-degree
  at most two \cite{Hendrickx_IJRNC2007}. Since removing one edge can
  reduce the out-degree of at most one of vertices $3$ and $5$, no
  minimally persistent spanning subgraph exists. Thus the directed
  graph is not constraint consistent, and hence is not
  persistent~\cite{Hendrickx_IJRNC2007}.

  Although the orientation fails the
  persistence criterion, the directed
  controller converges locally to the desired edge measurements, as
  shown in Figure~\ref{fig:ex7_edge}; one can also 
  check computationally that $-A$ is Hurwitz. Consequently, the 
  controller is locally exponentially stable at the target
  (Figure~\ref{fig:ex7_nodes}), in accordance with
  Theorem~\ref{thm:edge-to-node}. Thus, persistence is not necessary
  for local exponential stability of the directed distance-based controller.
\end{example}


\section{Edge Weight Design via SDP}\label{sec.sdp}
\begin{figure*}[t]
  \centering

  \begin{subfigure}[c]{0.28\textwidth}
    \centering
 \includegraphics[
  width=0.90\linewidth
]{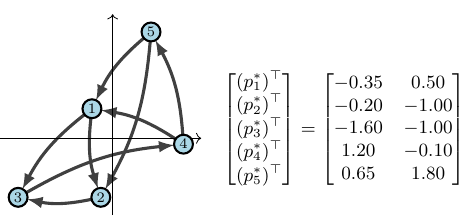}
    \caption{Directed target.}
    \label{fig:weight_design_graph}
  \end{subfigure}
  \hfill
  \begin{subfigure}[c]{0.15\textwidth}
    \centering
    \includegraphics[width=\linewidth]
    {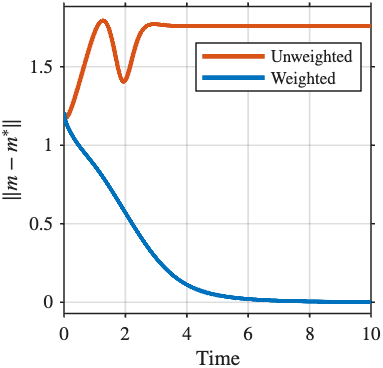}
    \caption{Edge-error norm.}
    \label{fig:weight_design_edge}
  \end{subfigure}
  \hfill
  \begin{subfigure}[c]{0.18\textwidth}
    \centering
    \includegraphics[width=\linewidth]
    {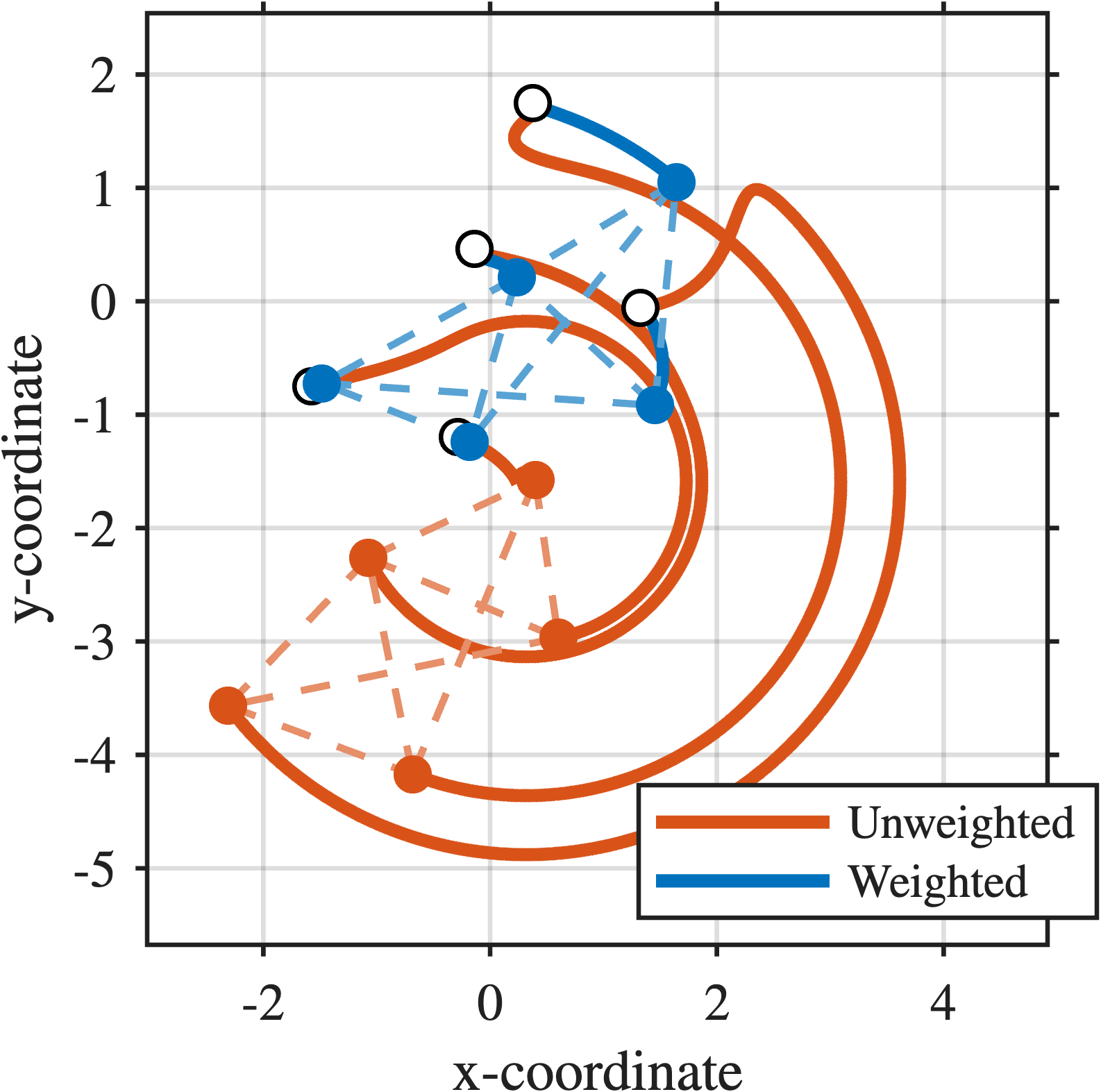}
    \caption{Node trajectories .}
    \label{fig:weight_design_nodes}
  \end{subfigure}

  \caption{
  Edge-weight synthesis (Example~\ref{ex:edge-weight-design}). The unweighted controller (red) is locally unstable; 
  the SDP-designed controller (blue) satisfies the quadratic certificate and converges.
  }
  \label{fig:edge-weight-design}
\end{figure*}
We have seen up to now that for directed sensing architectures, local
stability of a compatible controller can depend on both the target
geometry and the chosen orientation of the graph.  The stability
certificates of Theorem \ref{thm: convergence certificates}, on the
other hand, suggest that for a fixed target and graph orientation, it
may still be possible to shape the dynamics to recover desirable
properties.  A natural approach to accomplish this is to introduce a
scaled feedback law where each directed measurement is assigned an
independent gain.  


In this direction, we begin with the
compatible directed formation controller presented in
\eqref{dir_agent_grad_control}. We now associate a scalar gain
$w_{\overrightarrow{ij}}$ with each directed edge
$\overrightarrow{ij}$.  The corresponding weighted controller is
\[
  u_i = \sum_{\overrightarrow{ij}\in\E}
  w_{\overrightarrow{ij}}(\|p_i-p_j\|^2-m^*_{ij})(p_j-p_i).
\]
Thus, the weighting changes only the magnitude of each existing
feedback term while the sensing architecture and the information
available to each agent remain unchanged.

Let $W=\operatorname{diag}(w)\in\mathbb R^{|\E|\times |\E|}$ collect
the edge weights in the ordering used for the rows of
$\overrightarrow R(p)$.  The weighted controller can then be written
in aggregate form as
\[
  \dot p = -\overrightarrow{R}(p)^\top W(R(p)p-m^*).
\]
This defines a compatible formation controller in the sense of
Definition~\ref{def.compatform}. Indeed, setting
\[
  \nu(p,m)=\overrightarrow{R}(p)^\top W, \text{ and }
  \eta(p,m)=2R(p)\overrightarrow R(p)^\top W,
\]
gives $\eta(p,m)=2R(p)\nu(p,m)$. The corresponding edge dynamics are therefore
\[
  \dot m= 2R(p)\overrightarrow{R}(p)^\top W(m^*-m).
\]
For a fixed target $(p^*,m^*)$, let $Z:=2R(p^*)\overrightarrow
R(p^*)^\top$, as in Section \ref{sec.directed}.  If the columns of
$P$ form an orthonormal basis
for $T_{m^*}\mathcal Q$, then the linearization of the weighted edge
dynamics is
\[
  \dot{\delta m}=-A_W\delta m,
  \text{ with }
  A_W:=P^\top ZWP.
\]
By Theorem~\ref{thm: convergence certificates}, the weighted
controller is locally exponentially stable if $-A_W$ is Hurwitz.
A sufficient condition is $ \tfrac12(A_W+A_W^\top)\succ0$. Since
 $ \frac12(A_W+A_W^\top) =  \tfrac12  P^\top\left(ZW+WZ^\top\right)P$,
this condition is affine in the edge weights. 
%

%
Since the matrix inequality is affine in the optimization variables $w$, the
design problem can be formulated as the semidefinite program
\begin{equation}
  \begin{aligned}
    \max_{w,\gamma}\quad &\gamma\\
    \text{subject to}\quad
    &P^\top\!\left(ZW+WZ^\top\right)P
    \succeq
    \gamma I,\\
    & {\mathbf 1}^\top w = |\E|, \; W=\operatorname{diag}(w)\succeq 0.
  \end{aligned}
  \label{eq:weight_sdp}
\end{equation}
We restrict attention here to nonnegative edge gains and impose
$\mathbf{1}^\top w=|\E|$ to remove the arbitrary scaling of
the controller gains. The nonnegativity constraint is not required by the
stability condition itself. If signed edge weights are allowed, a
sign-independent normalization,
such as $\|w\|_1\leq |\E|$, may be used instead.

The optimal value $\gamma^\star$ gives the largest quadratic
stability margin under this gain normalization.
In particular, if $\gamma^\star>0$, then the synthesized edge weights
satisfy the sufficient condition of
Theorem~\ref{thm: convergence certificates} (b), and the corresponding
directed controller locally
exponentially stabilizes the target edge state.

We emphasize that infeasibility of
\eqref{eq:weight_sdp} does not preclude local exponential stability.
Rather, it indicates only that the stronger Lyapunov certificate of
Theorem~\ref{thm: convergence certificates} (b) cannot be established
through any choice of
\emph{diagonal} edge weights. The controller may nevertheless satisfy the
Hurwitz condition
of Theorem~\ref{thm: convergence certificates} (a).  In this sense,
\eqref{eq:weight_sdp} should be viewed as a constructive test for the
stronger quadratic Lyapunov certificate,
complementing the existence results of \cite{Yu_SIAMJCO2009} for
stabilizing diagonal weight matrices.
We conclude this section with an example.
\begin{example}[Stabilization by edge-weight design]
  \label{ex:edge-weight-design}
  Consider the directed target framework shown in
  Figure~\ref{fig:weight_design_graph}.
  With unit edge weights, the reduced operator  $A=P^\top ZP$ satisfies
  $\min \operatorname{Re}\lambda(A)=-0.1156<0$ and
  $\lambda_{\min}\!\left(\frac12(A+A^\top)   \right)=-0.9434$, so the
  unweighted directed controller is locally unstable and does not
  satisfy the quadratic certificate.

  We therefore solve the semidefinite program~\eqref{eq:weight_sdp}.
  The optimal value is $\gamma^\star=0.4926>0$, with edge weights
  \[
    {\small
      \setlength{\arraycolsep}{2.5pt}
      w^\star \approx
      \begin{bmatrix}
        0.244 & 0.110 & 1.058 & 2.665 &
        0.388 & 0.245 & 0.195 & 3.096
      \end{bmatrix}^{\!\top}.
    }
  \]
  For the resulting weighted operator $A_{W^\star}=P^\top ZW^\star P$,
  we obtain $\min \operatorname{Re}\lambda(A_{W^\star})
  =0.6993>0$ and $\lambda_{\min}\!\left(
    \frac12(A_{W^\star}+A_{W^\star}^\top)
  \right)
  =0.4926>0$. Hence the optimized weights satisfy the quadratic certificate
  of Theorem~\ref{thm: convergence certificates} (b) and locally
  exponentially stabilize the target.

  Figure~\ref{fig:weight_design_edge} compares the edge-error
  trajectories from the same initial condition, while
  Figure~\ref{fig:weight_design_nodes} shows the corresponding node
  trajectories. The unweighted controller moves away from the target
  and enters a limit cycle, whereas the weighted controller converges
  to a realization congruent to it.
\end{example}

\section{Further directions}\label{sec.conclusion}

Several questions remain open. The most immediate is to characterize when diagonal gains $W$ 
exist for which $-P^\top ZWP$ is Hurwitz, since our examples show Hurwitz stability 
can hold even when the quadratic certificate $P^\top(ZW+WZ^\top)P\succ0$ fails. It is also
natural to ask whether such gains admit a graph-theoretic characterization, 
and for which graph classes the Hurwitz and quadratic notions of stabilizability coincide. 
The other question of dynamical interest is whether, for every strongly admissible 
compatible controller, there is a non-empty open set of target configurations $p^*$ on which 
$-\eta^*|_{T_{m^*} \mathcal{Q}}$ is Hurwitz.

There are also several open questions concerning admissibility of 
directed controllers and directed orientations.  The most 
immediate is a combinatorial characterization of admissible and 
strongly admissible orientations of $d$-rigid graphs.  It would 
also be interesting to know for which directed graph classes
all the admissibility conditions coincide.


\setlength{\itemsep}{0pt}
\bibliographystyle{IEEEtran}
\bibliography{formation}

\begin{IEEEbiography}[{\includegraphics[
      width=1in,
      height=1.25in,
      clip,
      keepaspectratio
    ]{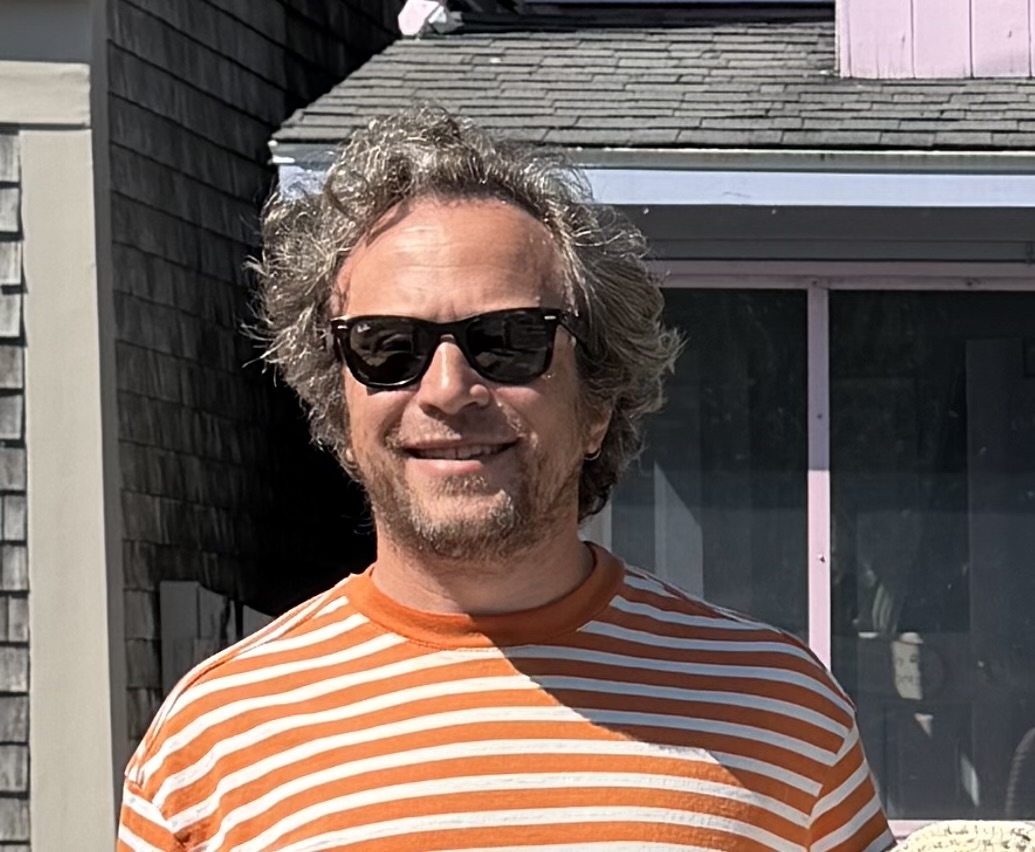}}]
{\bf Louis Theran} received BS, MS, and PhD degrees from the University of 
Massachusetts, Amherst, in 2006, 2007, and 2010.  He held postdoctoral
appointments at Temple University, Freie Universität Berlin and
an Aalto Science (ASci) Fellowship.
He is a  Lecturer in mathematics at the University of St Andrews.
Louis is a geometer who works on pure and applied problems, mainly around 
rigidity theory.
\end{IEEEbiography}
\begin{IEEEbiography}
[{\includegraphics[width=1in,height=1.25in,clip,keepaspectratio]{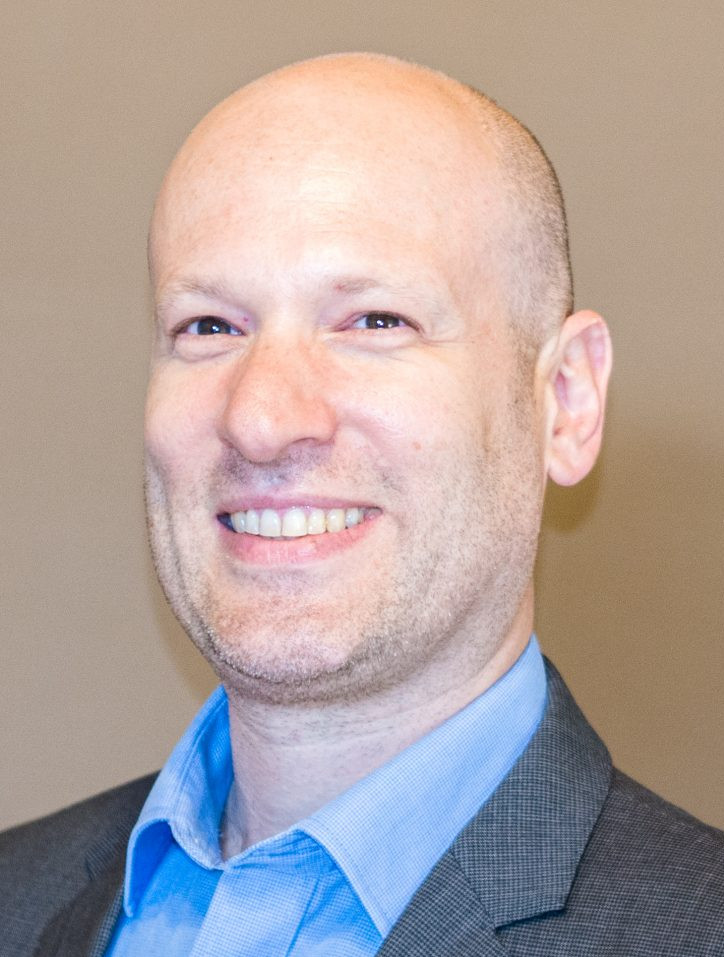}}]
  {\bf Daniel~Zelazo} (Senior Member, IEEE) received the B.Sc. and
  M.Eng. degrees in electrical engineering and computer science from
  the Massachusetts Institute of Technology, Cambridge, MA, USA, in
  1999 and 2001, respectively, and the Ph.D. degree in aeronautics and
  astronautics from the University of Washington, Seattle, WA, USA, in
  2009. From 2010 to 2012, he was a Postdoctoral Research Associate and
  Lecturer with the Institute for Systems Theory and Automatic Control,
  University of Stuttgart, Germany. He is a Professor of
  aerospace engineering with the Technion-Israel Institute of
  Technology, Haifa, Israel. His research interests include topics
  related to multiagent systems.
\end{IEEEbiography}
\begin{IEEEbiography}  [{\includegraphics[width=1in,height=1.25in,clip,keepaspectratio]{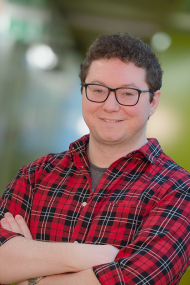}}]  {\bf Sean Dewar} is a Senior Postdoctoral Fellow at KU Leuven, Belgium.
  Prior to this position, he was a Research Scientist at RICAM,
  Austria, from 2019 to 2022, a Postdoctoral Fellow at the Fields
  Institute, Canada, in 2021, and a Heilbronn Fellow at the University
  of Bristol, UK, from 2022 to 2025.
  He received his Ph.D.~in Mathematics from Lancaster University, UK, in 2019.
  His research interests lie in applied algebraic geometry, discrete geometry and combinatorics. 
\end{IEEEbiography}
\begin{IEEEbiography} 
[{\includegraphics[width=1in,height=1.25in,clip,keepaspectratio]{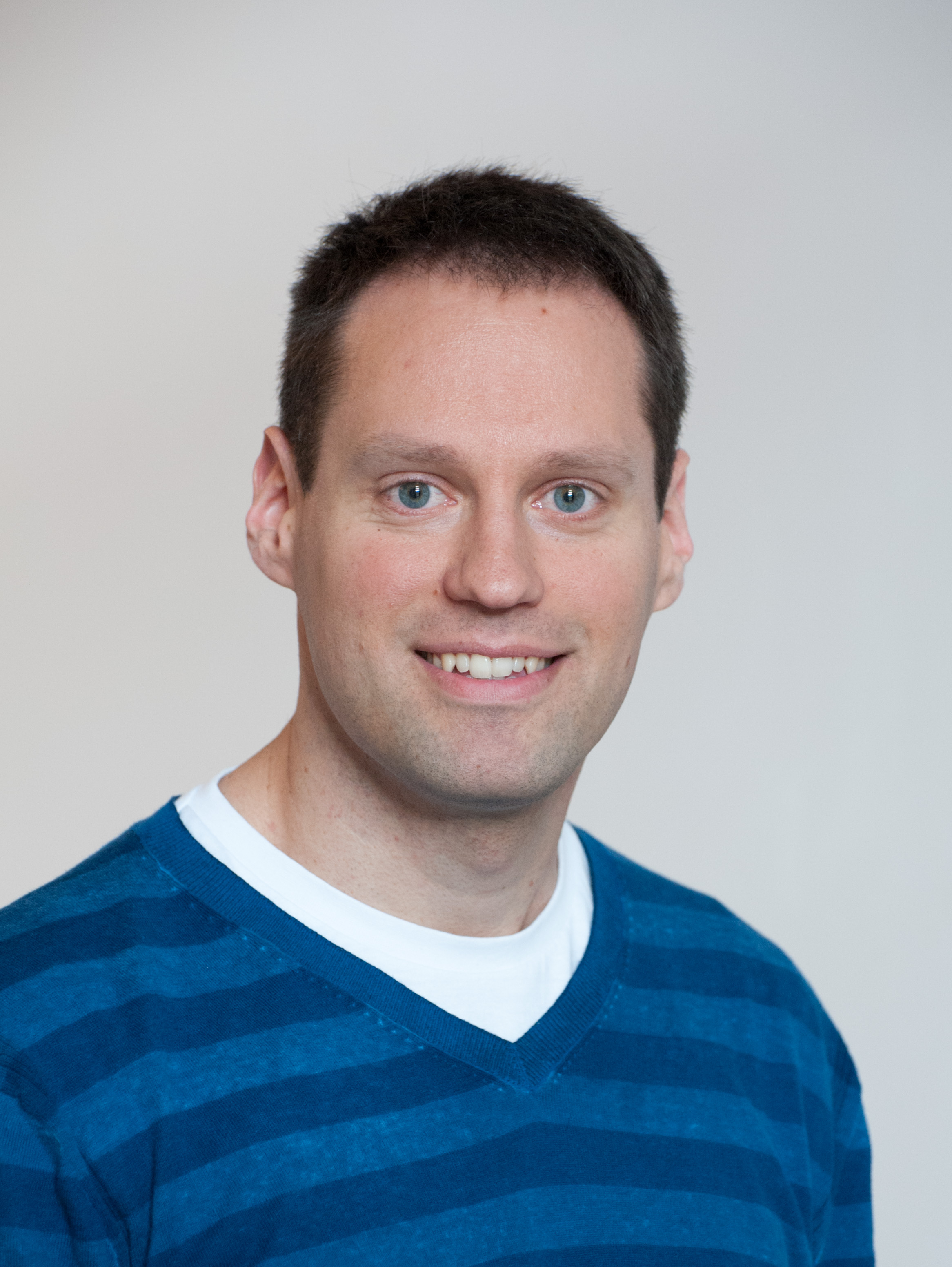}}] {\bf Bernd Schulze} is a Reader in Mathematics at Lancaster
  University, UK. Before arriving at Lancaster in 2012 he held
  postdoctoral positions at the TU Berlin, Germany, from 2009 to 2011
  and at the Fields Institute in Toronto, Canada, in 2011. Prior to
  that he studied mathematics and computer science at the FU Berlin,
  Germany, and at Western Michigan University, USA, and he received his
  Ph.D. in mathematics from York University, Canada, in 2009. 
  His 
  research interests lie in applied discrete geometry, 
  combinatorics, and algebraic methods in discrete mathematics.
\end{IEEEbiography}

\appendices
\makeatletter

\renewcommand{\@seccntformat}[1]{%
  \ifnum\pdfstrcmp{#1}{section}=0
    APPENDIX \csname the#1\endcsname.\quad
  \else
    \csname the#1\endcsname\quad
  \fi
}

\renewcommand{\section}{%
  \@startsection{section}{1}{\z@}%
  {0.7\baselineskip}%
  {0.25\baselineskip}%
  {\normalfont\normalsize\bfseries\centering}}

\makeatother
\section{Proof of Lemma \ref{lem: trapping}}
Let $x^*:=(p^*,m^*)$. Since $V$ is an open neighborhood of $x^*$ in
the smooth manifold $\mathcal S$, there is an $r>0$ such that the
closed ball $\overline W := \left\{(p,m)\in\mathcal
S:\|(p,m)-x^*\|\leq 2r\right\}$ is compact and contained in $V$. Let
$W$ denote the open ball that is the interior of $\overline W$. Since
$\nu$ is smooth, there exists $N>0$ such that $\|\nu_{(p,m)}\|\leq
N,$ for $(p,m)\in\overline W$.

Now, choose $T_0 > 0$ large enough that
$C\left(2+\frac{N}{c}\right)e^{-cT_0}<1$. We will use the continuous
dependence of the
trajectories on the finite interval $[0\;T_0]$, and then use the
exponential estimate \eqref{edge_expconv} to bound their
total displacement for $t\geq T_0$. Since $x^*$ is an equilibrium of
\eqref{eq:pm-dynamics}, there exists $0<\rho<r$ such that every
solution initialized at
  $x^0:=(p^0,m^0)\in
  U:=\{x\in\mathcal S:\|x-x^*\|<\rho\}$
satisfies $\|x(t)-x^*\|<r,
\; 0\leq t\leq T_0$. In particular, every such trajectory remains
in $W$ on $[0 \; T_0]$.

Fix an initial condition $x^0\in U$, and define the escape
time $T:=\sup\left\{\tau\geq T_0:
x(t)\in W \text{ for all } t\in[T_0,\tau]\right\}$. The supremum is over a set
containing $T_0$, so $T$ is well-defined.
Fix any $\tau$ with $T_0\leq\tau<T$. Since the trajectory remains in
$W\subseteq V$ for all $0\leq t\leq\tau$,
the estimate \eqref{edge_expconv} holds throughout this interval.

For $T_0\leq t\leq\tau$, the triangle inequality and
\eqref{edge_expconv} give
\[
  \begin{aligned}
    \|m(t)-m(T_0)\|
    &\leq \|m(t)-m^*\|+\|m(T_0)-m^*\| \\
    &\hspace{-1cm}\leq C\left(e^{-ct}+e^{-cT_0}\right)\|m^0-m^*\| \leq
    2Ce^{-cT_0}\rho,
  \end{aligned}
\]
where we used $t\geq T_0$ and $\|m^0-m^*\|\leq\|x^0-x^*\|<\rho$.

Next, compatibility and \eqref{edge_expconv} give
\[
  \begin{aligned}
    \|\dot p(t)\| &=\|\nu_{(p(t),m(t))}(m^*-m(t))\| \leq N\|m^*-m(t)\|\\
    &\leq NC e^{-ct}\|m^0-m^*\|<NC e^{-ct}\rho,
  \end{aligned}
\]
for $T_0\leq t\leq\tau$.  Integrating, we get
\[
  \|p(t)-p(T_0)\| \le \int_{T_0}^\infty \|\dot p(s)\| {\rm d}s \le
  (CN/c)e^{-cT_0}\rho.
\]
Combining these bounds, we obtain
\[
  \begin{aligned}
    \|x(t)-x(T_0)\|&= \|(p(t),m(t))-(p(T_0),m(T_0))\|\\
    &\leq \|p(t)-p(T_0)\|+\|m(t)-m(T_0)\|\\
    &\leq C\left(2+\frac{N}{c}\right)e^{-cT_0}\rho <\rho,
  \end{aligned}
\]
for $T_0\leq t\leq\tau$, where the last inequality follows from the
earlier choice of $T_0$.
Since $\|x(T_0)-x^*\|<r$, the triangle inequality gives
\[
  \|x(t)-x^*\| \leq \|x(t)-x(T_0)\|+\|x(T_0)-x^*\|<\rho+r<2r
\]
for all $T_0\leq t\leq\tau$. Thus the trajectory remains strictly inside
$W$ on $[T_0 \;\; \tau]$.

Suppose, for contradiction, that $T<\infty$. Since the estimate above
holds for every $\tau<T$, letting $\tau\to T$ from below and using
continuity of the trajectory gives
$\|x(T)-x^*\|\leq r+\rho<2r$. Hence $x(T)\in W$. Because $W$ is open
and the vector field is smooth,
the solution can be continued on some interval $[T\; T+\varepsilon)$
while remaining in $W$. This contradicts the definition of $T$ as the
escape time, and  therefore $T=\infty$. Combined with the
continuous-dependence bound on $[0\;\; T_0]$,
every solution initialized in $U$ remains in $W$ for all $t\geq 0$.
Thus $U\subseteq W\subseteq V$ satisfy the conclusion of the lemma.
\hfill $\qed$

\section{Proof of Lemma \ref{lem: rational maps generic properties}}
The characteristic polynomial $c_{\alpha({\bf t})}(x)$ has coefficients
in $\mathbb{Q}({\bf t})$.  By Cayley--Hamilton, the
transformation is invertible if and only if
the constant term $\det \alpha({\bf t})$ does not vanish.
If it vanishes identically,
then $\alpha({\bf t})$ is generically non-invertible.
Otherwise, $\det \alpha({\bf t})\neq 0$, so  has the form $p({\bf
t})/q({\bf t})$
with $p,q\in \mathbb{Q}[{\bf t}]$ both non-zero.  Away from the set of
${\bf t}$ such that $p({\bf t})$ or $q({\bf t})$ vanish, $\alpha({\bf t})$
is invertible, and, hence, $\alpha({\bf t})$ is generically invertible. This
is statement (a).

Statement (b) is similar, except we add to the exceptional set
the numerator and product of the denominators in the resultant
$\operatorname{Res}(c_{\alpha({\bf t})}(x),c_{\alpha({\bf t})}(-x);x)$,
which vanishes exactly when $\alpha({\bf t})$ has a pair of eigenvalues
that sum to zero.
When $\alpha({\bf t}_0)$ is Hurwitz for some specialization ${\bf t}_0$, then
the resultant is generically non-vanishing, since the sum of any
pair of eigenvalues of $\alpha({\bf t}_0)$  has negative real part.
\hfill$\qed$

\section{Proof of Lemma \ref{lem: rational maps compression}}
Let $A\in \mathbb{R}({\bf t})^{m\times m}$ be a matrix for the
family $\alpha({\bf t})$ in the standard basis, and suppose that
the rank of $A$ is $r$. By Gaussian elimination, there is a
basis $\{v_1, \ldots, v_r, w_1, \ldots, w_{m-r}\}$ of
$\mathbb{R}({\bf t})^m$ so that the $v_i$ span $\IM \alpha({\bf t})$.
Hence, the projection $\pi({\bf t})$ to $\IM \alpha({\bf t})$ is also
rational in
${\bf t}$, as is the inclusion $\pi^*({\bf t}): \IM \alpha({\bf
t})\hookrightarrow \mathbb{R}({\bf t})^m$.
Since the image is an invariant subspace, the restriction is
given by $\pi({\bf t})\circ \alpha({\bf t})\circ \pi^*({\bf t})$, which,
as a composition of rational maps is rational.
\hfill $\qed$.

\section{Proof of Lemma \ref{lem: decoupling}}
That (a) implies (b) is immediate from the definitions.
For (b) implies (a), let $v\in V_j$ be given and suppose that
$v = \alpha(w)$.  Among the pre-images $w$, choose one such that
$w\in V_{j'}$ with $j'$ minimal.  If $j' > j$, (b) implies that
there is a $w'\in V_{j'-1}$ such that $\alpha(w) = \alpha(w')$.
This contradicts minimality of $j'$, so (b) implies (a).

Now we address (c). From the
intertwining relation $\pi_j\circ \alpha = \alpha_j\circ\pi_j$ and
surjectivityy of $\pi_j : V_j\to W_j$, we have
$\pi_j(\alpha(V_j)) = \IM \alpha_j$.  Since the
kernel of $\pi_j\circ \alpha|_{V_j} = \alpha(V_j)\cap V_{j-1}$,
rank-nullity gives $\rk \alpha|_{V_j} = \rk \alpha_j +
\dim(\alpha(V_j)\cap V_{j-1})$.
By invariance of $V_{j-1}$, $\alpha(V_{j-1})\cap V_{j-1}\subseteq
\alpha(V_j)\cap V_{j-1}$,
and so we get the inequality $\rk \alpha|_{V_j} \ge  \rk \alpha_j +
\rk \alpha|_{V_{j-1}}$.
The inequality is an equality if and only if
$\alpha(V_{j-1})\cap V_{j-1} =  \alpha(V_j)\cap V_{j-1}$. This latter
condition is (b).

Finally (d) is equivalent to (c) by telescoping: (d) can hold if and only
if each of the inequalities from above
$\rk \alpha|_{V_j} \ge  \rk \alpha_j + \rk \alpha|_{V_{j-1}}$
is an equality.
\hfill $\qed$.

\section{Proof of Lemma \ref{lem: flag}}
The spectral statement is standard and follows from the
fact that the characteristic polynomial
$c_\alpha(x) = c_{\alpha_1}(x)\cdots c_{\alpha_k}(x)$.

Statement (a) follows from the fact that
Lemma \ref{lem: decoupling} (c) expresses the
decoupling condition in terms of ranks of maps that
are rational in {\bf t} and their restrictions to
rationally defined spaces.  Hence, the decoupling
condition is a generic property by Lemmas
\ref{lem: rational maps generic properties} and
\ref{lem: rational maps compression}.
The proof of (b) is similar, noting that
\eqref{eq: skew} is equivalent to the statement that
$\rk \alpha^2 = \rk \alpha$.\hfill $\qed$

\section{Proof of Lemma \ref{lem: flag skew}}
We first suppose that \eqref{eq: skew} does not
hold for $\alpha$; i.e., that there is a non-zero
$v\in \IM \alpha\cap \ker \alpha$. Let $j$ be minimal
such that $v\in V_j$; in particular $v\notin V_{j-1}$,
so $\pi_j(v)\neq 0$.
By the decoupling condition and Lemma \ref{lem: decoupling} (b), 
there is a $w\in V_j$ such that $\alpha(w) = v$.  We
now obtain, from the intertwining
relation $\pi_j\circ \alpha = \alpha_j \circ \pi_j$ on $V_j$, that 
$\pi_j(v) = \pi_j(\alpha(w)) = \alpha_j(\pi_j(w))$,
which shows that $\pi_j(v)\in \IM \alpha_j$.  We also
have, because $v\in \ker \alpha$,  $0 = \pi_j(\alpha(v)) = \alpha_j(\pi_j(v))$.
Hence, $\pi_j(v) \in \IM \alpha_j\cap \ker \alpha_j$ 
and is non-zero, so \eqref{eq: skew} fails for $\alpha_j$.

Now suppose that for some $j$, \eqref{eq: skew} does hold for
$\alpha$.  For convenience, set $M = \IM \alpha$ and
$M_j = M\cap V_j$.  The decoupling condition
says that $M_j = \alpha(V_j)$.  Injectivity of 
$\alpha$ on $M$ descends to $M_j$, and so 
$\alpha(M_j)\subseteq M_j$ is promoted to 
the equality $\alpha(M_j) = M_j$.  This yields
a slight strengthening of the decoupling 
condition: if $v\in M_j$, then there is a 
$u\in M_j$ (not just $V_j$) such that $\alpha(u) = v$.

Fix a $j$ and
let $w\in \IM \alpha_j\cap \ker \alpha_j$ be given.
By the intertwining relation
$\IM \alpha_j = \pi_j(\alpha(V_j)) = \pi_j(M_j)$,
so there is a $v\in M_j$ such that $\pi_j(v) = w$.
Since $w\in \ker \alpha_j$,
intertwining gives $0 = \alpha_j(w) = \pi_j(\alpha(v))$.
Hence, $\alpha(v)\in M_{j-1}$.  From above, there is a 
$u\in M_{j-1}$ such that $\alpha(u) = \alpha(v)$.
Injectivity of $\alpha$ on $M_j$ implies that $u = v$, and 
so $0 = \pi_j(u) = \pi_j(v) = w$.  This shows that
\eqref{eq: skew} holds for $\alpha_j$.
\hfill $\qed$

\section{Proof of Lemma \ref{lem: multiaffine}}\label{pf_lemma3}
If $A$ vanishes identically, it does so on every ${\bf x}\in \{0,1\}^m$.
For the other direction, the key observation is that, because $A$
is multiaffine, the variable supports of its monomials are
distinct.  Suppose that $A\neq 0$, and let
$M$ be a monomial with support of minimum
cardinality.  Setting every variable in the support of $M$ to be
one and the rest of the variables to be zero produces an
${\bf x}\in \{0,1\}^m$ such that $A({\bf x}) = M({\bf x}) \neq 0$, as required.
\hfill $\qed$

\section{Proof of Lemma \ref{lem: weighted determinant}}\label{pf_lemma4}
Using Cauchy--Binet twice, we compute
\begin{align*}
  \det B^\top A({\bf t})B &
  = \sum_S \det B^\top[-,S]\det (A({\bf t})B)[S,-] \\
  &\hspace{-1cm} = \sum_S \sum_T \det B[S,-]\det B[T,-] \det A({\bf t})[S,T],
\end{align*}
where $S$ and $T$ are subsets of $[m]$ with cardinality $p$, and
$A[S,T]$ denotes the submatrix with
rows indexed by $S$ and columns indexed by $T$, and $B[S,-]$ denotes
the submatrix formed from the
rows indexed by $S$.  This exhibits the determinant on the l.h.s. as a linear
(the minors of $B$ are scalars) combination of $p\times p$ minors of
$A({\bf t})$,
which are multiaffine by hypothesis. Multiaffine polynomials are
closed under linear combinations,
giving the desired statement.
\hfill $\qed$

\end{document}